\documentclass[a4paper,num-refs]{oup-contemporary-spidec}

\journal{general}
\jlogo{}

\usepackage{graphicx}
\usepackage{siunitx}
\usepackage{amssymb}
\usepackage{amsmath}
\usepackage{amsthm}
\usepackage{amscd}
\newtheorem{theorem}{Theorem}[section]

\newtheorem{lemma}[theorem]{Lemma}
\newtheorem{remark}[theorem]{Remark}
\newtheorem{assumption}[theorem]{Assumption}
\newtheorem{definition}[theorem]{Definition}
\usepackage{bm}
\usepackage{bigints}
\usepackage{amsthm}

\newtheorem{testcase}{Test} 

\usepackage[colorlinks]{hyperref}
\hypersetup{
    linkcolor=jcolour,
    citecolor=jcolour,
    urlcolor=jcolour
}
\usepackage{orcidlink}

\newenvironment{acknowledgement}{\par\addvspace{17pt}\small\rmfamily\noindent}{\par\addvspace{6pt}}

\title{Stable Positive Integral Deferred Correction Methods for Positive Dynamical Systems}

\author[1,2,\authfn{1},\authfn{2},\orcidlink{0000-0002-1869-945X}]{Mario Pezzella}

\affil[1]{CNR National Research Council of Italy -- Institute for Applied Mathematics ``M. Picone". Via P. Castellino 111, Naples 80131, Italy} 
\affil[2]{University of Naples Federico II, Department of Mathematics and Applications “Renato Caccioppoli”, Via Cintia, Naples, 80126, Italy}

\authnote{\authfn{1}Corresponding author \href{mailto:mario.pezzella@cnr.it}{mario.pezzella@cnr.it}} 
\authnote{\authfn{2} Member of the \emph{Gruppo Nazionale Calcolo Scientifico} (GNCS) of the \emph{Istituto Nazionale di Alta Matematica} (INdAM), the \emph{Research Italian network on Approximation} (RITA), and the SIMAI Group \emph{Numerical and Analytical Approximation of Data and Functions with Applications} (ANA\&A).}

\papercat{Research Paper Preprint}

\runningauthor{Mario Pezzella}

\jyear{2026}

\begin{document}

\begin{frontmatter}
\maketitle

\begin{abstract}
    In this paper, we introduce the class of \emph{Stable Positive Integral Deferred Correction} (SPIDeC) methods for the numerical integration of positive dynamical systems. The proposed framework embeds a deferred correction mechanism within an exponential-type Volterra reformulation of the underlying differential problem. The resulting multiplicative structure guarantees the unconditional preservation of both positivity and equilibria, independently of the integration stepsize. Arbitrarily high-order accuracy is systematically achieved through successive explicit-in-sweep corrections applied to a low-order base approximation. From a stability viewpoint, the SPIDeC integrators are L-stable and exactly reproduce the continuous semigroup generated by diagonal linear operators. Furthermore, when Gauss--Radau quadrature nodes are employed, the associated discrete flow asymptotically approaches a logarithmically contractive map as the number of sweeps increases, ensuring stability. Numerical experiments are provided to validate the theoretical analysis and illustrate the practical performance of the proposed methods.
\end{abstract}

\begin{keywords}
$\bullet$ Unconditional positivity $\ \bullet$ Structure-preserving $\ \bullet$ Geometric integration $\ \bullet$ Exponential Volterra equation $\ \bullet$ Deferred correction $\ \bullet$ Positive ODEs
\end{keywords}

\end{frontmatter}

\vspace{-1\baselineskip}

\section{Introduction}
In this work, we address initial value problems of the form
\begin{equation}\label{eq:ODE_system}
\begin{cases}
    \bm{y}'(t) = \bm{f}(\bm{y}(t)), \qquad  t \in [0,T],\\
    \bm{y}(0) = \bm{y}_0 \in \mathbb{R}_{+}^{N},
\end{cases}
\end{equation}
arising either directly as ordinary differential equation (ODE) models or from the spatial discretization of partial differential equations (PDEs). Here, $\bm{y}(t) = [y_1(t),\dots,y_N(t)]^\top$ and
$\bm{f} = [f_1,\dots,f_N]^\top$, with $\bm{f} : \Omega \to \mathbb{R}^N$ sufficiently smooth function defined on an open set containing $\mathbb{R}_{+}^{N}=\{\bm{\eta}=[\eta_1,\dots,\eta_N]^\top \, : \ \eta_i>0, \ \forall i=1,\dots,N\}$. Our analysis is carried out under the following structural property.
\begin{assumption}\label{ass:Positivity}
    The positive orthant $\mathbb{R}_{+}^{N}$ is forward invariant under the flow of \eqref{eq:ODE_system}. In particular, the exact solution corresponding to any initial condition $\bm{y}_0 \in \mathbb{R}_{+}^{N}$ satisfies
    \begin{equation*}
        \bm{y}(t) \in \mathbb{R}_{+}^{N}, \quad \text{for all } t \in [0,T],
    \end{equation*}
    so that all components remain strictly positive throughout the time interval.
\end{assumption}
Assumption~\ref{ass:Positivity}, which often reflects physical or biological interpretations of the model variables, is inherently satisfied in many applications including population dynamics, epidemiological modeling and chemical reaction networks. From a mathematical viewpoint, it can be guaranteed under verifiable conditions on $\bm{f}$. A classical sufficient condition requires that $f_i(\bm{\eta})>0$ for any $\bm{\eta}\in\mathbb{R}^N_{\geq 0}$ satisfying $\eta_i=0$, which induces an inward-pointing behavior on the boundary of the positive orthant. A concise characterization is provided by the Kamke–M{\"u}ller condition, according to which the associated flow preserves the partial order induced by the positive cone if and only if each component $f_i(\bm y)$ is non-decreasing with respect to $y_j$ for all $j\neq i$ (see, for instance, \cite{Hanouzet_1996,Smith_order_preserving_survey}). Consequently, if $\bm f(\bm 0)\ge\bm0$, positivity follows by comparison with the zero solution. Alternatively, a structural representation $f_i(\bm{y}) = y_i g_i(\bm{y})$, with $g_i$ locally Lipschitz continuous, ensures invariance of each coordinate hyperplane and preservation of positivity. Additional results for specific ODE frameworks can be found in \cite[Theorem 3.3]{Formaggia_2011} and \cite[Theorem 1.2]{Torlo_Issues}.

The numerical integration of positive dynamical systems poses additional challenges, since standard discretizations do not, in general, inherit the geometric invariance properties of the underlying continuous problem (see \cite{hairer2006geometric} and references therein). As a consequence, even high order methods may produce non-physical, non-positive approximations unless sufficiently restrictive stepsizes are enforced. Furthermore, post-processing strategies such as clipping or projection, which truncate negative components, may introduce artificial mass, degrade accuracy and adversely affect stability \cite{Blanes_2022,N_lein_2021}. These considerations motivate the development of unconditionally positivity-preserving integrators, which retain positivity independently of the time steplength. In this context, several non-standard finite difference (NSFD) methods have been proposed in the literature \cite{computation12090183,NSFD_Dajana,NSFD_JCD,NSFD_MBE,Mickens_2000}. While these schemes preserve the dynamical consistency of the underlying continuous models, their construction is inherently problem-dependent, and extensions beyond first order remain less straightforward (see, for instance, \cite{Arenas03062026,Hoang_2024,HOANG2026130029}). For the specific class of production–destruction systems, the Modified Patankar (MP) approach provides a systematic framework for constructing accurate,  unconditionally positive and conservative schemes. Building on the original formulation in \cite{BURCHARD20031}, the MP methodology has been extended to Runge–Kutta \cite{AVILA2021113350,Huang_2023,Izgin_2025}, multistep \cite{Izzo_2025}, semi-Lagrangian \cite{Cacace_2026} and deferred correction \cite{Torlo_Issues,Torlo_Offner} settings. However, its formulation relies intrinsically on the production-destruction structure and does not naturally extend to general positive systems of the form \eqref{eq:ODE_system}. Recently, high order positivity-preserving methods have been developed by combining suitable exponential or rational Volterra-type reformulations of the differential operator with predictor–corrector \cite{CdS,Replicator_Pezzella} or fully implicit direct quadrature \cite{Mixing,Axioms,Pezzella_ESAIM} discretizations. Such non-linear formulations circumvent the classical limitations of linear positivity-preserving schemes established in \cite{Bolley_1978}.

In this paper, we develop the class of \emph{Stable Positive Integral Deferred Correction} (SPIDeC) methods, which embed the Deferred Correction (DeC) paradigm \cite{Abgrall_2017, Dutt_2000, Ong_2020} within a Volterra–exponential reformulation of \eqref{eq:ODE_system}. The resulting framework can be interpreted as a discrete analogue of the classical Picard–Lindelöf theory, where the solution is obtained as the limit of successive Picard iterations applied to an equivalent integral equation. In this spirit, the method starts from a low-order approximation and refines it through successive correction sweeps. In contrast to classical DeC approaches, which act additively on the differential formulation, the proposed construction relies on a multiplicative integral representation. This structure yields unconditionally positive stable approximations of arbitrary order through explicit-in-sweep stage equations.

The paper is structured as follows. Section \ref{sec:SPIDeC} outlines the construction of SPIDeC integrators and provides additional details on the employed spectral quadrature. A comprehensive analysis of the structure-preserving properties of the proposed methods is presented in Section \ref{sec:StructurePreserving}. The high-order consistency and convergence of the SPIDeC framework are studied in Section \ref{sec:Error_Analysis}. Section \ref{sec:Stability} is devoted to an investigation of linear and non-linear stability. Numerical experiments are reported in Section \ref{sec:Numerics}, while concluding remarks and future research directions are given in Section \ref{sec:Conclusion}.

\section{The class of SPIDeC methods}\label{sec:SPIDeC}
Let $t_n = n h,$ $n=0,\dots,N_t,$ denote a uniform partition of the interval $[0,T]$ with stepsize $h>0.$ Consider, for each subinterval $[t_n,t_{n+1}]$, the set of quadrature nodes
\begin{equation}\label{eq:nodi_quad}
    t_{n,m} = t_n + \tau_m h, \qquad m=0,\dots,M \quad \text{where} \quad 0 \leq \tau_0 < \tau_1 < \dots < \tau_M = 1.
\end{equation}
Then, for each time step, each stage and each correction sweep, the general \emph{Stable Positive Integral Deferred Correction} (SPIDeC) numerical method is defined component-wise as
\begin{equation}\label{eq:SPIDeC}
    \begin{aligned}
        & y^{n,m}_{i(0)} = y_i^n
        \exp\!\left(
        h\,\tau_m \,
        \dfrac{f_i(\bm{y}^n)}{y_i^n}
        \right), 
        &&\begin{array}{l}
            i=1,\dots,N,  \\
            n \!=0,\dots,N_t, 
        \end{array}
         \\
        & y^{n,m}_{i(k)} = y_i^n
        \exp\!\left(
        h \sum_{j=0}^M Q_{m,j}\,
        \dfrac{f_i\!\left(\bm{y}^{n,j}_{(k-1)}\right)}{y^{n,j}_{i(k-1)}}
        \right), 
        &&\begin{array}{l}
            \!m\!=0,\dots,M,  \\
            k=1,\dots,N_k,
        \end{array} 
        \\
        & y^{n+1}_i=y^{n,M}_{i(N_k)}, \ \ \  \qquad \qquad \qquad \qquad \qquad \qquad \qquad \qquad && \begin{array}{l}
            n=0,\dots,N_t-1,
        \end{array} 
    \end{aligned}
\end{equation}
where $y_i^0=y_i(0)\in \mathbb{R}^+$ is given, $y_i^n\approx y_i(t_n)$ and $y_{i(k)}^{n,m}\approx y_i(t_{n,m})$ is the stage approximation at the correction sweep $k.$ When $\tau_0=0$ (e.g. Gauss--Lobatto nodes), the first stage coincides with the initial value of the time step. At sweep $k$, the stage values $\{y^{n,m}_{i(k)}\}_{m=0}^M$ depend explicitly only on those at sweep $k-1$, and are therefore mutually independent. As a consequence, all $M+1$ stages can be computed in parallel, without intra-sweep communication.  The quantities $Q_{m,j}$ denote quadrature weights associated to the nodes $\{\tau_j\}_{j=0}^M$ and will be specified later on (cf. Section \ref{subsec:Quadrature}). 

The non-linear multiplicative structure of \eqref{eq:SPIDeC} reflects an underlying exponential formulation of the differential system \eqref{eq:ODE_system}. Indeed, under the Assumption  \ref{ass:Positivity}, the exact solution satisfies the following implicit Volterra Integral Equation (VIE) 
\begin{equation}\label{eq:VIE}
    y_i(t) = y_i(t_n)
    \exp\!\left(
    \bigintsss_{\, t_n}^{t}
    \dfrac{f_i(\bm{y}(s))}{y_i(s)}\, ds\right), \qquad i=1,\dots,N, \quad n=0,\dots,N_t-1, \quad t \in [t_n,t_{n+1}],
\end{equation}
which naturally follows by integration over each subinterval. In this context, the SPIDeC methods \eqref{eq:SPIDeC} can be regarded as high order approximations of \eqref{eq:VIE}, enhanced through a deferred correction mechanism acting on the integral term. More specifically, equation \eqref{eq:SPIDeC}$_1$ corresponds to a first order prediction of the solution that is iteratively refined through the correction sweeps \eqref{eq:SPIDeC}$_2$ up to the final stage which is used for the step update \eqref{eq:SPIDeC}$_3$. 

Starting from the VIE \eqref{eq:VIE}, the formulation \eqref{eq:SPIDeC} can be rigorously derived within
the framework of Deferred Correction (DeC) methods. Following the arguments in \cite{Abgrall_2017}, the
construction relies on two operators: a low–order operator
$\bm{\mathcal L}^1$, which is inexpensive to invert, and a high–order operator $\bm{\mathcal L}^2$, which represents the desired accurate discretization. The correction sweeps are then defined by the DeC iteration
\begin{equation}\label{eq:DeC_Iteration}
    \bm{\mathcal L}^1(\bm Y_{(k)}) =\bm{\mathcal L}^1(\bm Y_{(k-1)})-\bm{\mathcal L}^2(\bm Y_{(k-1)}), \qquad k=1,\dots,N_k ,
\end{equation}
starting from an initial predictor obtained from
$\bm{\mathcal L}^1(\bm Y_{(0)})=0$. By construction, if a fixed point of \eqref{eq:DeC_Iteration} exists, it naturally satisfies the highly accurate discrete problem $\bm{\mathcal L}^2(\bm Y)=0$.

Here, the operators in \eqref{eq:DeC_Iteration} are derived from the Volterra integral equation \eqref{eq:VIE}. For each time step we collect the stage values in the matrix $\bm{Y}^{n} =
\big[ \bm{y}^{n,0},\dots,\bm{y}^{n,M} \big]
\in \mathbb{R}^{N\times(M+1)}.$ The low–order operator $\bm{\mathcal{L}}^1$ is obtained by approximating the integral in \eqref{eq:VIE} by a first–order left-rectangle rule as follows
\begin{equation}\label{eq:L1}
    \mathcal{L}_i^1(\bm{Y}^{n})=\begin{pmatrix}
        y_i^{n,0}-y_i^{n}\exp\!\left(\displaystyle\int_{t_n}^{t_{n,0}} \mathcal{E}_0^{-1}[\bm{F}^i](s) \, ds \right) \\
        \vdots \\
        y_i^{n,M}-y_i^{n}\exp\!\left(\displaystyle\int_{t_n}^{t_{n,M}} \mathcal{E}_0^{-1}[\bm{F}^i](s) \, ds \right)
    \end{pmatrix}^\top=\begin{pmatrix}
        y_i^{n,0}-y_i^{n}\exp\!\left(h\,\tau_0 \, \dfrac{f_i(\bm y^{n})}{y_i^{n}}\right) \\
        \vdots \\
        y_i^{n,M}-y_i^{n}\exp\!\left(h\,\tau_M \, \dfrac{f_i(\bm y^{n})}{y_i^{n}}\right)
    \end{pmatrix}^\top, \qquad \qquad i=1,\dots,N,
\end{equation}
where $\bm{F}^i=[f_i(\bm{y}^{n,0}),\dots,f_i(\bm{y}^{n,M})]^\top\oslash[y_i^{n,0},\dots,y_i^{n,M}]^\top\in \mathbb{R}^{M+1},$ $\oslash$ denotes the component-wise division and $\mathcal{E}_q^{-1}[\bm{F}^i]$ represents the 
polynomial of degree $q\in\mathbb{N}$ interpolating the data $\bm{F}^i$ at the  nodes $\tau_0,\dots,\tau_q$ (cf.\ Section~\ref{subsec:Quadrature}). Specifically, $\mathcal{E}_0^{-1}[\bm{F}^i]$ denotes the constant taking the value $f_i(\bm{y}^n)/y_i^n$, which corresponds to the left-rectangle approximation of the integrand evaluated at the left endpoint $t_n$ of the time interval, independently of the location of $\tau_0$.

Similarly, the high–order operator $\bm{\mathcal{L}}^2$ is devised component-wise by  
\begin{equation}\label{eq:L2}
    \mathcal{L}_i^2(\bm{Y}^{n})=\begin{pmatrix}
        y_i^{n,0}-y_i^{n}\exp\!\left(\displaystyle\int_{t_n}^{t_{n,0}} \mathcal{E}_M^{-1}[\bm{F}^i](s) \, ds \right) \\
        \vdots \\
        y_i^{n,M}-y_i^{n}\exp\!\left(\displaystyle\int_{t_n}^{t_{n,M}} \mathcal{E}_M^{-1}[\bm{F}^i](s) \, ds \right)
    \end{pmatrix}^\top=\begin{pmatrix}
        y_i^{n,0}-y_i^{n}\exp\!\left( h \sum_{j=0}^{M} Q_{0,j} \, \dfrac{f_i(\bm y^{n,j})}{y_i^{n,j}} \right) \\
        \vdots \\
        y_i^{n,M}-y_i^{n}\exp\!\left( h \sum_{j=0}^{M} Q_{M,j}\, \dfrac{f_i(\bm y^{n,j})}{y_i^{n,j}} \right)
    \end{pmatrix}^\top, \quad \quad i=1,\dots,N,
\end{equation}
where $Q_{mj}$ denote the quadrature weights defined in Section \ref{subsec:Quadrature}. Solving $\bm{\mathcal L}^1(\bm Y_{(0)}^n)=0$ directly gives the predictor stage \eqref{eq:SPIDeC}$_1.$ Substituting $\bm{\mathcal L}^1$ and $\bm{\mathcal L}^2$ into the DeC iteration \eqref{eq:DeC_Iteration} yields the correction sweeps
\eqref{eq:SPIDeC}$_2$.

The SPIDeC scheme $\bm Y_{(k)}^n =\bm Y^n_{(k-1)}-(\bm{\mathcal L}^1)^{-1}(\bm{\mathcal L}^2(\bm Y^n_{(k-1)}))$ can be therefore interpreted as a
preconditioned Picard iteration, with $\bm{\mathcal L}^1$ acting as the preconditioner, to solve the non-linear
collocation system 
\begin{equation}\label{eq:coll_syst_SenzaL}
    \log(y_i^{n,m}) =\log(y_i^{n})+h\sum_{j=0}^{M} Q_{m,j} \dfrac{f_i(\bm y^{n,j})}{y_i^{n,j}}, \qquad i=1,\dots,N, \qquad m=0,\dots,M,
\end{equation}
in close analogy with the interpretation of classical
Spectral Deferred Correction (SDC) methods (see, for instance, \cite{Dutt_2000,Ong_2020,Speck_2018}). However, in contrast to SDC integrators, where corrections are typically written in additive form, the SPIDeC scheme performs multiplicative exponential corrections. This formulation naturally preserves positivity of the numerical solution (cf. Section \ref{sec:StructurePreserving}) and inherits the order–raising mechanism characteristic of deferred correction iterations (see Section \ref{sec:Error_Analysis}). 

\begin{remark}
    For simplicity of exposition, we restrict our analysis to the autonomous case. However, the SPIDeC scheme~\eqref{eq:SPIDeC} extends straightforwardly to non-autonomous systems 
    $\bm{y}'(t)=\bm{f}(\bm{y}(t),t)$ by replacing 
    $f_i(\bm{y}^{n,j}_{(k-1)})$ with $f_i(\bm{y}^{n,j}_{(k-1)},t_{n,j})$. 
    The structure-preserving properties established in the following sections, as well as the convergence and stability properties, do not rely on the autonomy of the vector field and therefore remain valid, under analogous regularity assumptions, in the non-autonomous setting.
\end{remark}

\subsection{Spectral quadrature and discrete integration operator}\label{subsec:Quadrature} 
Here, we provide details on the construction of the quadrature  weights $Q_{m,j}$, which proceeds through polynomial interpolation  of the integrand, its exact integration via an operatorial  formulation and the stable numerical assembly of the resulting  matrix. From now on, given an integer $r,$ we denote by $\mathbb{P}_r$ the space of polynomials of degree at most $r.$

The starting point is the approximation of the integral term in \eqref{eq:VIE} for $t=t_{n,m}$. Specifically, by the straightforward change of variable $s=t_n+\tau h,$ it is expressed, for each $i=1,\dots,N,$ as follows
\begin{equation*}
    \bigintsss_{\, 0}^{\tau_m}
    \dfrac{f_i(\bm{y}(t_n+\tau h))}{y_i(t_n+\tau h)}\, d\tau, \qquad  \quad \begin{array}{l}
         n=0,\dots,N_t-1,  \\
         m=0,\dots,M 
    \end{array}
\end{equation*}
and then approximated by replacing the integrand function with the unique polynomial in $\mathbb{P}_M$ interpolating its nodal values at the points $\{\tau_j\}_{j=0}^M$. Specifically, given the nodal data
\begin{equation*}
    \bm{F}^i
    =\bigl[f_i(\bm{y}^{n,0}),\dots,f_i(\bm{y}^{n,M})\bigr]^\top
     \oslash
     \bigl[y_i^{n,0},\dots,y_i^{n,M}\bigr]^\top
    \in\mathbb{R}^{M+1},
\end{equation*}
where $\oslash$ represents the component-wise division, the interpolant reads
\begin{equation*}
    p(\tau)
    = \mathcal{E}_M^{-1}[\bm{F}^i](\tau)
    = \sum_{j=0}^M F^i_j\,\ell_j(\tau),
    \qquad \tau\in[0,1],
\end{equation*}
where $\{\ell_j\}_{j=0}^M$ is the Lagrange basis associated with $\{\tau_j\}_{j=0}^M$, and $\mathcal{E}_M^{-1}$ denotes the interpolation operator that maps nodal values to the unique element of $\mathbb{P}_M$ attaining them. Integrating $p(\tau)$ exactly over $[0,\tau_m]$ then gives
\begin{equation}\label{eq:quad_derivation}
    \int_0^{\tau_m} p(\tau)\,d\tau
    = \sum_{j=0}^M F^i_j
      \int_0^{\tau_m}\ell_j(\tau)\,d\tau,
    \qquad \ \ \  m=0,\dots,M,
\end{equation}
which defines the quadrature weights
\begin{equation}\label{eq:forma_pesi}
    Q_{m,j} = \int_0^{\tau_m}\ell_j(\tau)\,d\tau,
    \qquad m,j=0,\dots,M.
\end{equation}
Replacing the integral in~\eqref{eq:VIE} with the right-hand side of~\eqref{eq:quad_derivation} directly produces the high order operator~$\bm{\mathcal{L}}^2$ in~\eqref{eq:L2} and the correction sweeps~\eqref{eq:SPIDeC}$_2$.

The construction above admits a compact operatorial description. Define the Volterra integral operator and the nodal evaluation operator as
\begin{equation*}
    \mathcal{V}:\mathbb{P}_M\to\mathbb{P}_{M+1},
    \qquad
    (\mathcal{V}p)(\tau)=\int_0^\tau p(s)\,ds,
    \qquad\qquad
    \mathcal{N}:C([0,1])\to\mathbb{R}^{M+1},
    \qquad
    (\mathcal{N}f)_j=f(\tau_j).
\end{equation*}
The restriction $\mathcal{N}_M=\mathcal{N}|_{\mathbb{P}_M}$ is a linear isomorphism by the unisolvence of polynomial interpolation at distinct nodes~\cite[Thm.~8.1]{Quarteroni_2007}. The discrete quadrature operator is then defined by the composition
\begin{equation*}
    \mathcal{Q}
    = \mathcal{N}_{M+1}\circ\mathcal{V}\circ\mathcal{N}_M^{-1}
    : \mathbb{R}^{M+1}\to\mathbb{R}^{M+1},
\end{equation*}
which maps nodal values of a polynomial to the nodal values of its primitive. By construction, the diagram
\begin{equation}\label{Diagramma}
    \begin{CD}
        \mathbb{P}_M @>\mathcal{V}>> \mathbb{P}_{M+1} \\
        @A\mathcal{N}_M^{-1}AA @VV\mathcal{N}_{M+1}V \\
        \mathbb{R}^{M+1} @>>\mathcal{Q}> \mathbb{R}^{M+1}
    \end{CD}
\end{equation}
commutes, and the entries of $\mathcal{Q}$ coincide with the weights
in~\eqref{eq:forma_pesi}. In particular, $\mathcal{Q}$ realizes the exact action of $\mathcal{V}$ on $\mathbb{P}_M$ in nodal coordinates, that is for every
$p\in\mathbb{P}_M$,
\begin{equation*}
    \int_0^{\tau_m}p(\tau)\,d\tau
    = \sum_{j=0}^M Q_{m,j}\,p(\tau_j),
    \qquad m=0,\dots,M,
\end{equation*}
so that the quadrature rule~\eqref{eq:forma_pesi} is polynomially exact of degree $M$. In particular, applying the exactness property to $p\equiv 1\in\mathbb{P}_M$ gives $\sum_{j=0}^M Q_{m,j}=\tau_m$ for each $m=0,\dots,M.$ Moreover, for sufficiently  smooth integrands the quadrature error is of order $\mathcal{O}(h^{M+1})$, with spectral convergence for analytic functions when Chebyshev nodes are employed. The role of $M$ in the global accuracy of SPIDeC is analyzed in Section~\ref{sec:Error_Analysis}.

\begin{remark}[On the computation of the quadrature matrix $Q$]
Equation~\eqref{eq:forma_pesi} provides a theoretical characterization of the quadrature weights. However, their direct evaluation via explicit Lagrange polynomials in the monomial basis is numerically unstable for large $M$, due to the well-known ill-conditioning of this representation. To avoid this issue, the matrix $Q$ is assembled using the \texttt{Chebfun} environment~\cite{Driscoll2014ChebfunGuide}. Each Lagrange basis function $\ell_j(\tau)$ is constructed via the \texttt{chebfun.interp1} routine, which implements the barycentric interpolation formula in a numerically stable manner~\cite{BerrutTrefethen2004}. The Volterra integral operator $\mathcal{V}$ is then applied through the \texttt{cumsum} command, which computes antiderivatives spectrally on the Chebyshev representation of the interpolant~\cite{AurentzTrefethen2017}. Evaluation at the quadrature nodes $\{\tau_m\}$ then yields the columns of $Q$.

Since the quadrature nodes remain fixed throughout the simulation, the matrix $Q$ is computed once in a preprocessing step and reused at every time iteration, ensuring both accuracy and computational efficiency.
\end{remark}

\section{Structure-Preserving Properties of SPIDeC Methods}
\label{sec:StructurePreserving}
In this section, we establish the structure-preserving properties of the SPIDeC scheme~\eqref{eq:SPIDeC}. We first prove that positivity is preserved unconditionally with respect to $h$, independently of the number of correction sweeps, with all stage values strictly positive. We then show that the method exactly preserves the equilibrium set of the continuous system while maintaining high order accuracy, as established in Section~\ref{sec:Error_Analysis}. 
\begin{theorem}[Unconditional positivity]\label{thm:Positivity_Unc}
    Consider the discrete equations~\eqref{eq:SPIDeC} and assume that the
    initial values $y_i^0=y_i(0)$ are positive for each $i=1,\dots,N$.
    Then, independently of the stepsize $h>0$, the relations
    \begin{equation}\label{eq:Num_Positivity}
        y^{n,m}_{i(k)}>0, \qquad y_i^n>0, \qquad \text{for all} \quad i=1,\dots,N, \quad m=0,\dots,M, \quad k=0,\dots,N_k,
    \end{equation}
    hold for each $n=0,\dots,N_t$.
\end{theorem}
\begin{proof}
    The statement follows directly from the multiplicative-exponential structure of the scheme. Indeed, for any fixed $n$, both the predictor and the correction stages in~\eqref{eq:SPIDeC} define the internal values $y^{n,m}_{i(k)}$ as the product of $y_i^n$ and an exponential term, which is strictly positive for any real argument. Hence, positivity of $y_i^n$ implies positivity of all stage values and, in particular, of $y_i^{n+1}=y^{n,M}_{i(N_k)}$. The result then follows by standard induction on $n$.
\end{proof}

\subsection{Preservation of equilibria}
Here, we address the correspondence between stationary solutions of the continuous system~\eqref{eq:ODE_system} and fixed points of the time-marching scheme~\eqref{eq:SPIDeC}. We denote by
\begin{equation}\label{eq:Phi}
    \bm{\Phi}^{h,N_k} : \mathbb{R}_{+}^{N} \longrightarrow \mathbb{R}_{+}^{N}, \qquad \bm{\Phi}^{h,N_k}(\bm{y}^n) = \bigl[\Phi^{h,N_k}_1(\bm{y}^n),\dots,\Phi^{h,N_k}_N(\bm{y}^n)\bigr]^\top, \qquad \Phi_i^{h,N_k}(\bm{y}^n) = y^{n,M}_{i(N_k)}, \quad i=1,\dots,N,
\end{equation}
the numerical map induced by one step of~\eqref{eq:SPIDeC} with $M+1$ stages, $N_k$ correction sweeps and stepsize $h$. We further define the set of equilibria of~\eqref{eq:ODE_system} and the set of fixed points of the SPIDeC scheme as
\begin{equation}\label{eq:Sets}
    \mathcal{E} = \bigl\{\bm{y}^*\in\mathbb{R}_{+}^{N} \;:\; f_i(\bm{y}^*)=0, \;\;\forall\, i=1,\dots,N\bigr\} \quad \text{and} \quad \mathcal{F}^{h,N_k} =\bigl\{\bm{y}^*\in\mathbb{R}_{+}^{N} \;:\; \bm{\Phi}^{h,N_k}(\bm{y}^*)=\bm{y}^*\bigr\},
\end{equation}
respectively.

For standard Runge--Kutta methods, every equilibrium of the continuous system is a fixed point of the discrete flow, while the converse generally fails. As a result, spurious non-physical fixed points may arise as discretization artifacts~\cite{Hairer_Iserles_SanzSerna_1990,Humphries_1993}. For the SPIDeC method~\eqref{eq:SPIDeC}, the characterization of $\mathcal{F}^{h,N_k}$ depends on the number of correction sweeps $N_k$. We first establish the inclusion $\mathcal{E}\subseteq\mathcal{F}^{h,N_k}$ unconditionally in $h>0$, $M\geq 0$ and $N_k\geq 0$. We then prove that, for $N_k=0$, the converse inclusion holds as well, yielding the exact characterization $\mathcal{E}=\mathcal{F}^{h,0}$. For $N_k>0$, the analysis is carried out at the level of the underlying collocation problem, for which an analogous characterization is established under mild structural assumptions on $f$.

\begin{theorem}[Unconditional equilibria preservation]
\label{thm:EqSubsetFP}
    Let $\bm{y}^*\in\mathcal{E}$. Then $\bm{\Phi}^{h,N_k}(\bm{y}^*)=\bm{y}^*$ for every $h>0$, $M\geq 0$ and $N_k\geq 0$. Equivalently, $\mathcal{E}\subseteq\mathcal{F}^{h,N_k}$.
\end{theorem}

\begin{proof}
    Let $\bm{y}^n=\bm{y}^*$. We prove by induction on $k=0,\dots,N_k$ that
    \begin{equation}\label{eq:Induct}
        y^{n,m}_{i(k)}=y_i^* \qquad \forall\, i=1,\dots,N, \quad \forall\, m=0,\dots,M.
    \end{equation}
    Regarding the base case of $k=0,$ the equilibrium condition $f_i(\bm{y}^*)=0$ gives, via~\eqref{eq:SPIDeC}$_1$,
    \begin{equation*}
        y^{n,m}_{i(0)} = y_i^*\exp\!\left(h\,\tau_m\,\frac{f_i(\bm{y}^*)}{y_i^*}\right) = y_i^*\exp(0)=y_i^*, \qquad \qquad \begin{array}{l}
             i=1,\dots,N,  \\
             m=0,\dots,M. 
        \end{array}
    \end{equation*}
    Consider now $k>0$ and suppose $y^{n,j}_{i(k-1)}=y_i^*$ for all $i=1,\dots,N$ and $j=0,\dots,M$. The correction sweep~\eqref{eq:SPIDeC}$_2$ gives
    \begin{equation*}
        y^{n,m}_{i(k)}
        = y_i^*\exp\!\left(h\sum_{j=0}^{M} Q_{m,j}\,
          \frac{f_i\!\left(\bm{y}^{n,j}_{(k-1)}\right)}{y^{n,j}_{i(k-1)}}\right)
        = y_i^*\exp\!\left(h\sum_{j=0}^{M} Q_{m,j}\,
          \frac{f_i(\bm{y}^*)}{y_i^*}\right)
        = y_i^*\exp(0)=y_i^*,
        \qquad \qquad \begin{array}{l}
             i=1,\dots,N,  \\
             m=0,\dots,M. 
        \end{array}
    \end{equation*}
    This establishes~\eqref{eq:Induct} for all $k$. Taking $m=M$ and $k=N_k$, we obtain $\Phi_i^{h,N_k}(\bm{y}^*)=y^{n,M}_{i(N_k)}=y_i^*$ for all $i$.
\end{proof}

The following result concerns the SPIDeC scheme~\eqref{eq:SPIDeC} in the absence of correction sweeps.
\begin{theorem}[Fixed-point characterization for the predictor scheme]
\label{thm:Fix_Point_Corr_0}
    Let $N_k=0$. Then, for every $h>0$ and $M\geq 0$, a vector $\bm{y}^*\in\mathbb{R}_{+}^{N}$ is a fixed point of the SPIDeC scheme~\eqref{eq:SPIDeC} if and only if it is an equilibrium of the ODE system~\eqref{eq:ODE_system}. Equivalently, $\mathcal{E}=\mathcal{F}^{h,0}$.
\end{theorem}

\begin{proof}
    Sufficiency follows from Theorem~\ref{thm:EqSubsetFP}. For necessity, let
    $\bm{y}^*\in\mathcal{F}^{h,0}$. Since $N_k=0$ and $\tau_M=1$, the
    numerical flow reduces to the predictor step~\eqref{eq:SPIDeC}$_1$
    evaluated at the last stage
    \begin{equation*}
        \Phi_i^{h,0}(\bm{y}^*)
        = y_i^*\exp\!\left(h\,\frac{f_i(\bm{y}^*)}{y_i^*}\right),
        \qquad \quad i=1,\dots,N.
    \end{equation*}
    The fixed-point condition $\bm{\Phi}^{h,0}(\bm{y}^*)=\bm{y}^*$, together
    with $y_i^*>0$ and the injectivity of the exponential function, gives
    \begin{equation*}
        \exp\!\left(\frac{h\,f_i(\bm{y}^*)}{y_i^*}\right)=1
        \quad\iff\quad
        f_i(\bm{y}^*)=0,
        \qquad \quad \forall\, i=1,\dots,N,
    \end{equation*}
    where the last equivalence uses $h>0$ and $y_i^*>0$.
\end{proof}

In order to address the case $N_k>0$, we work at the level of the non-linear collocation system associated to the operator $\bm{\mathcal{L}}^2$ in \eqref{eq:L2}. By construction of the SPIDeC framework, a stationary point of the correction iteration~\eqref{eq:DeC_Iteration}
satisfies the high order collocation problem
$\bm{\mathcal{L}}^2(\bm{Y})=\bm{0}$.
In what follows, we assume the following global unique solvability condition.
\begin{assumption}
\label{ass:Solv}
     For each $h>0$ and each initial datum $\bm{y}^n\in\mathbb{R}_{+}^{N}$, the collocation system
    \begin{equation}\label{eq:coll_syst_Con_L}
            \bm{\mathcal{L}}^2(\bm{Y}^n)=\bm{0},   \\
        \qquad \text{where} \qquad \bm{Y}^{n} =
        \big[ \bm{y}^{n,0},\dots,\bm{y}^{n,M} \big]=\begin{bmatrix}
        y_1^{n,0} & \dots & y_1^{n,M} \\
        \vdots & \vdots  & \vdots \\
        y_N^{n,0} & \dots & y_N^{n,M} \\ 
        \end{bmatrix},
    \end{equation}
    with $\bm{y}^n$ entering as a known parameter on the 
    right-hand side of~\eqref{eq:L2}, admits a unique 
    solution $\tilde{\bm{Y}}^n\in\mathbb{R}^{N\times(M+1)}_{>0}$.
\end{assumption}

The collocation system~\eqref{eq:coll_syst_Con_L} comprises $N(M+1)$ equations for the $N(M+1)$ unknowns $y_i^{n,m}$, with $\bm{y}^n$ treated as a given parameter. When $\tau_0=0$, the first equation reduces to $y_i^{n,0}=y_i^n$ (since $Q_{0,j}=0$ for all $j$), so that the first stage is fixed and the system reduces to $NM$ non-trivial equations. When $\tau_0>0$, all $M+1$ equations remain non-trivial and the full $N(M+1)$ system characterizes the collocation solution. In both cases, expressing the system in the logarithmic variables $\lambda_i^m=\log(y_i^{n,m}/y_i^n)$ and evaluating the Jacobian at $\lambda_i^m=0$ (the trivial stage $y_i^{n,m}=y_i^n$ for all $m$) yields the identity matrix for $h=0$. By the implicit function theorem, the system has a unique solution for all $h\in[0,\hat{h})$ for some $\hat{h}>0$. Assumption~\ref{ass:Solv} requires this to hold globally for all $h>0$.

Under the Assumption~\ref{ass:Solv}, the \emph{collocation map} is well defined as
\begin{equation}\label{eq:CollocationMap}
    \bm{\Psi}^h : \mathbb{R}_{+}^{N} \to \mathbb{R}_{+}^{N},
    \qquad
    \bm{\Psi}^h(\bm{y}^n) =
    \bigl[\Psi_1^h(\bm{y}^n),\dots,\Psi_N^h(\bm{y}^n)\bigr]^\top,
    \qquad
    \Psi_i^h(\bm{y}^n)=\tilde{y}_i^{n,M},
    \quad i=1,\dots,N,
\end{equation}
where $\tilde{\bm{Y}}^n$ is the unique solution of~\eqref{eq:coll_syst_Con_L}.
By construction, $\bm{\Psi}^h$ represents the ideal limit of $\bm{\Phi}^{h,N_k}$ as $N_k\to\infty$ and, for finite $N_k$, the map $\bm{\Phi}^{h,N_k}$ approximates $\bm{\Psi}^h$ with an error governed by the convergence rate of the DeC iteration~\eqref{eq:DeC_Iteration}. We accordingly define the set of \emph{collocation fixed points} as
\begin{equation*}
    \mathcal{C}^{h}
    =\bigl\{\bm{y}^*\in\mathbb{R}_{+}^{N}
      \;:\; \bm{\Psi}^{h}(\bm{y}^*)=\bm{y}^*\bigr\}.
\end{equation*}

The following result represents the collocation counterpart of Theorem \ref{thm:EqSubsetFP}.
\begin{theorem}[Equilibria as collocation fixed points]
\label{thm:EqSubsetFPCOLL}
    Let $\bm{y}^*\in\mathcal{E}$ and suppose Assumption~\ref{ass:Solv} holds.
    Then $\bm{\Psi}^{h}(\bm{y}^*)=\bm{y}^*$ for every $h>0$. Equivalently,
    $\mathcal{E}\subseteq\mathcal{C}^{h}$.
\end{theorem}
\begin{proof}
    Let $\bm{y}^n=\bm{y}^*$. By Theorem~\ref{thm:EqSubsetFP} (applied for any
    $N_k\geq 0$), the constant stage matrix $\bm{Y}^*=[\bm{y}^*,\dots,\bm{y}^*]$
    satisfies $\bm{\mathcal{L}}^2(\bm{Y}^*)=\bm{0}$ (as $f_i(\bm{y}^*)=0$
    implies that each exponential argument in~\eqref{eq:L2} vanishes). Since $\bm{Y}^*$ satisfies the collocation equations with $\bm{y}^n=\bm{y}^*$ as the known parameter, the uniqueness
    guaranteed by Assumption~\ref{ass:Solv} implies $\tilde{\bm{Y}}^n=\bm{Y}^*$.
    Hence $\bm{\Psi}^h(\bm{y}^*)=\tilde{y}_i^{n,M}=y_i^*$ for all $i$.
\end{proof}

In what follows, a full characterization of discrete equilibria is established for the two-stages version of the SPIDeC integrator~\eqref{eq:SPIDeC}.
\begin{theorem}[Collocation fixed-point characterization for $M=1$ and $\tau_0=0$]
\label{thm:FP_Collocation_M1}
    Consider the SPIDeC scheme~\eqref{eq:SPIDeC} with $M=1$ and $\tau_0=0$. Let Assumption~\ref{ass:Solv} hold. Then, for every $h>0$, a point $\bm{y}^*\in\mathbb{R}_{+}^{N}$ is a collocation fixed point of $\bm{\Psi}^h$ if and only if it is an equilibrium of~\eqref{eq:ODE_system}. Equivalently, $\mathcal{E}=\mathcal{C}^{h}$.
\end{theorem}
\begin{proof}
    The inclusion $\mathcal{E}\subseteq\mathcal{C}^{h}$ is given by
    Theorem~\ref{thm:EqSubsetFPCOLL}. For the converse, let
    $\bm{\Psi}^h(\bm{y}^*)=\bm{y}^*$, i.e., let the unique collocation solution
    with initial condition $\tilde{\bm{y}}^{n,0}=\bm{y}^n=\bm{y}^*$ satisfy
    $\tilde{y}_i^{n,1}=y_i^*$. Since $M=1$, the collocation system~\eqref{eq:coll_syst_Con_L} consists of a
    single equation per component 
    \begin{equation}\label{eq:Colloc_M1_gen}
        \tilde{y}_i^{n,1}
        = y_i^*\exp\!\left(h\sum_{j=0}^{1} Q_{1,j}\,
          \frac{f_i\!\left(\tilde{\bm{y}}^{n,j}\right)}
               {\tilde{y}_i^{n,j}}\right), \qquad i=1,\dots,N.
    \end{equation}
    Given $\tilde{y}_i^{n,0}=\tilde{y}_i^{n,1}=y_i^*$, both growth rates
    appearing in~\eqref{eq:Colloc_M1_gen} coincide and using $\sum_{j=0}^1 Q_{1,j}=\tau_1=1$
    (cf.\ Section~\ref{subsec:Quadrature}), equation~\eqref{eq:Colloc_M1_gen}
    reduces to
    \begin{equation*}
        y_i^* = y_i^*\exp\!\left(\frac{h\,f_i(\bm{y}^*)}{y_i^*}\right), \qquad i=1,\dots,N.
    \end{equation*}
    The positivity $y_i^*>0$, the condition $h>0$ and the injectivity of the exponential then yield $f_i(\bm{y}^*)=0$ for all $i=1,\dots,N,$ which completes the proof.
\end{proof}

For $M\geq 2$, the interior stages $\tilde{y}_i^{n,m}$ ($m=1,\dots,M-1$) are free and their growth rates enter the last collocation equation as a weighted balance. For large $h$, such contributions may cancel even when $f_i(\bm{y}^*)\neq 0$, giving rise to spurious collocation fixed points. The following theorem identifies a sufficient condition on the structure of $\bm{f}$ that rules out this scenario for all $h>0$.
\begin{theorem}[Collocation fixed-point characterization]
\label{thm:FP_Colloc_General}
    Let $M\geq 2$ and suppose Assumption~\ref{ass:Solv} holds together with:
    \begin{enumerate}[label=\normalfont(\roman*)]
        \item \label{it:weights}\textbf{(Positive quadrature weights)}
              $Q_{M,j}>0$ for all $j=0,\dots,M$;
        \item \label{it:sign}\textbf{(Sign consistency of growth rates)}
              For each $i=1,\dots,N$ and for every collocation fixed point
              $\bm{y}^*\in\mathcal{C}^h$, the growth rate
              $\bm{y}\mapsto f_i(\bm{y})/y_i$ has the same sign along the collocation trajectory, that is
              \begin{equation}\label{eq:SignCondition}
                  \frac{f_i\!\left(\tilde{\bm{y}}^{n,j}\right)}
                       {\tilde{y}_i^{n,j}}
                  \cdot
                  \frac{f_i(\bm{y}^*)}{y_i^*}
                  \;\geq\; 0,
                  \qquad j=0,\dots,M.
              \end{equation}
    \end{enumerate}
    Then, for every $h>0$, each collocation fixed point is an equilibrium of
    ~\eqref{eq:ODE_system}. Equivalently, $\mathcal{E}=\mathcal{C}^{h}$.
\end{theorem}
\begin{proof}
    The inclusion $\mathcal{E}\subseteq\mathcal{C}^h$ follows 
    from Theorem~\ref{thm:EqSubsetFPCOLL}. Let 
    $\bm{y}^*\in\mathcal{C}^h$, so that the collocation 
    solution with initial datum $\bm{y}^n=\bm{y}^*$ 
    satisfies $\tilde{y}_i^{n,M}=y_i^*$ (fixed-point 
    condition). The collocation equation at $m=M$ reads
    \begin{equation}\label{eq:FixedPt_Constraint}
        \sum_{j=0}^{M} Q_{M,j}\,
        \frac{f_i\!\left(\tilde{\bm{y}}^{n,j}\right)}
             {\tilde{y}_i^{n,j}}
        = 0,
        \qquad i=1,\dots,N.
    \end{equation}
    Setting $\alpha_i^j = f_i(\tilde{\bm{y}}^{n,j})/
    \tilde{y}_i^{n,j}$, the sign 
    condition~\ref{it:sign} ensures that every $\alpha_i^j$ 
    has the same sign as 
    $\alpha_i^M = f_i(\bm{y}^*)/y_i^*$, where the 
    identification $\tilde{\bm{y}}^{n,M}=\bm{y}^*$ follows 
    from the fixed-point condition. Since $Q_{M,j}>0$ by 
    assumption~\ref{it:weights}, 
    equation~\eqref{eq:FixedPt_Constraint} is a sum of 
    terms with the same sign equal to zero. Therefore each 
    summand vanishes, and in particular 
    $Q_{M,M}\,\alpha_i^M=0$. Since $Q_{M,M}>0$, we 
    conclude $\alpha_i^M = f_i(\bm{y}^*)/y_i^* = 0$, 
    and hence $f_i(\bm{y}^*)=0$ for all $i=1,\dots,N$.
\end{proof}
    Condition~\ref{it:weights} is satisfied by Gauss--Lobatto and Gauss--Legendre nodes, whose quadrature weights are strictly positive~\cite{Quarteroni_2007}. It is also satisfied by right Gauss--Radau nodes, which confer unconditional logarithmic contractivity to the SPIDeC framework (as proved in Section \ref{subsec:NonlinStab}). For these nodes, no stage is located at $\tau=0$, so the sum~\eqref{eq:FixedPt_Constraint} runs from $j=1$ to $j=M$, and condition~\ref{it:sign} is required only along the interior collocation trajectory $\{\tilde{\bm{y}}^{n,j}\}_{j=1}^M$.
    
    Condition~\ref{it:sign} reflects a non-restrictive structural property of the underlying dynamics. It is verified in several relevant settings, including Lotka--Volterra systems under dissipativity near a stable equilibrium, and compartmental models with suitable inflow and outflow terms and sign-definite per-capita rates.
    Outside these regimes, condition~\ref{it:sign} may fail for sufficiently large $h$, and spurious collocation fixed points cannot be excluded a priori.

The results of this section reveal a hierarchy across correction levels. At the predictor stage ($N_k=0$), $\mathcal{E}=\mathcal{F}^{h,0}$ for all $h>0$, so no spurious fixed points arise. As $N_k$ increases, the numerical flow $\bm{\Phi}^{h,N_k}$ converges to the collocation map $\bm{\Psi}^h$, for which $\mathcal{E}=\mathcal{C}^h$ holds under mild assumptions. Hence, correction sweeps improve the order of accuracy (Section~\ref{sec:Error_Analysis}) without altering the equilibrium structure. This decoupling distinguishes the SPIDeC framework from standard high order Runge--Kutta methods, where additional stages may introduce non-physical fixed points.

\section{Error Analysis}\label{sec:Error_Analysis}
The SPIDeC scheme \eqref{eq:SPIDeC} consists of a prediction step, a sequence of correction sweeps, and a final update. Accordingly, the local discretization error can be decomposed as
\begin{equation}\label{eq:Local_error}
    \begin{aligned}
        &\delta^{n,m}_{i(0)}(h)=y_i(t_{n,m})-y_{i(0)}(t_{n,m}), &&\text{with} \ \ y_{i(0)}(t_{n,m}):=y_i(t_n)\exp \left(h \tau_m \dfrac{f_i(\bm{y}(t_n))}{y_i(t_n)}\right), \\
        &\delta^{n,m}_{i(k)}(h)=y_i(t_{n,m})-y_{i(k)}(t_{n,m}),    &&\text{with} \ \ y_{i(k)}(t_{n,m})=y_i(t_n)\exp \left(h \sum_{j=0}^M Q_{m,j} \dfrac{f_i(\bm{y}_{(k-1)}(t_{n,j}))}{y_{i(k-1)}(t_{n,j})}\right), \quad k=1,\dots,N_k, \\
        & \delta_{i}(h; t_{n+1})=y_i(t_{n+1})-y_{i(N_k)}(t_{n+1})=\delta^{n,M}_{i(N_k)}(h),\!\!\!
    \end{aligned}
\end{equation}
for $i=1,\dots,N$ and $m=0,\dots,M$. For convenience, we also introduce the vector notation
\begin{equation}\label{eq:Vec_Local_error}
    \bm{\delta}_{(k)}^{n,m}(h) = [\delta_{1(k)}^{n,m}(h),\dots,\delta_{N(k)}^{n,m}(h)]^\top \in \mathbb{R}^N, \qquad k=0,\dots,N_k.
\end{equation}

The multiplicative--exponential structure of the scheme, which ensures unconditional positivity, requires a careful treatment of the error propagation. In the following, we show that the SPIDeC method \eqref{eq:SPIDeC} achieves arbitrarily high order of accuracy, determined by the number of quadrature nodes and correction sweeps. The consistency of the scheme is stated in the following theorem.
\begin{theorem}[SPIDeC Consistency]\label{thm:Consistency}
    Let $\bm{y}(t)=[y_1(t),\dots,y_N(t)]^\top, $ for $t\in[0,T],$ denote the solution to the differential system \eqref{eq:ODE_system} under the Assumption \ref{ass:Positivity}. Given two integers $M$ and $N_k,$ define $q=\max\{M+1, \ N_k+1\}.$ Assume that $\bm f \in C^{q}(\Omega)$, where $\Omega$ is a compact subset of $\mathbb{R}_{+}^{N}$ such that
    \begin{equation}\label{eq:def_Omega}
         \left\{\bm{\omega}=[\omega_1,\dots,\omega_N]^\top \ : \ 0<\alpha\leq \omega_i \leq \beta\right\}\subset \stackrel{\circ}{\Omega}, \qquad \text{with} \ \ \alpha=\!\min_{\substack{1\leq i \leq N \\ 0\leq t \leq T }}|y_i(t)| \ \ \text{and} \ \  \beta =\! \max_{i=1,\dots,N}\|y_i(t)\|_{L^\infty([0,T])}.
    \end{equation}
    Then, the SPIDeC method \eqref{eq:SPIDeC} with $M+1$ stages and $N_k$ correction sweeps is consistent with \eqref{eq:ODE_system} of order $p=\min\{M+1,N_k+1\}$.
\end{theorem}
\begin{proof}
    We introduce, for each $i=1,\dots,N$, the functions $\varphi_i\in C^q(\Omega)$ as 
    $\varphi_i : \bm{\omega}\in \Omega \mapsto f_i(\bm{\omega})/\omega_i \in \mathbb{R}$. 
    Since the trajectory $\bm y(t)$ remains in the compact set $\Omega$ for all $t\in[0,T]$, the functions $\varphi_i$ admit uniformly bounded partial derivatives up to order $q$. For $r=0,\dots,q$, we then define $\hat{\Phi}_i^{[r]}=\max_{\bm{\omega}\in\Omega}\max_{|\gamma|=r}\left|D^\gamma \varphi_i(\bm{\omega})\right|$ and $\hat{\Phi}^{[r]}=\max_i\hat{\Phi}_i^{[r]}.$ \\
    To prove the result, we proceed by induction on $k$ showing that, independently of $n$ and $m,$
    \begin{equation}\label{eq:Induction_Objective}
        \|\bm{\delta}^{n,m}_{(k)}(h)\|_\infty=\mathcal{O}(h^{M+2})+\mathcal{O}(h^{k+2}), \qquad \text{for each} \ k=0,\dots,N_k.
    \end{equation}
    Regarding the base case $k=0$, which corresponds to the predictor stage in \eqref{eq:SPIDeC}${_1},$ we first recall that the left-endpoint quadrature rule yields
    \begin{equation*}
        \left|R^{n,m}_i\right|=\left|\int_{t_n}^{t_{n,m}}\varphi_i(\bm{y}(s)) \ ds - h\tau_m\varphi_i(\bm{y}(t_n))\right|\leq C_i^m \tau_m^2 h^2\leq \mathrm{C}_\mathrm{R} h^2, \qquad \begin{array}{l}
             i=1,\dots,N,  \\
             m=0,\dots,M, 
        \end{array}
    \end{equation*}
    with $C_i^{m}$ and $\mathrm{C}_\mathrm{R}$ positive constants independent of $h.$
    Furthermore, from \eqref{eq:VIE} and the mean value theorem, there exist some constants $\zeta_i^m\in(0,e^{h\tau_m\hat{\Phi}_i^{[0]}})$ such that
    \begin{equation*}
    \begin{split}
        \delta^{n,m}_{i(0)}(h)=&y_i(t_n)\left\{ \exp\left(\int_{t_n}^{t_{n,m}}\varphi_i(\bm{y}(s)) \ ds\right) - \exp\left(h\tau_m\varphi_i(\bm{y}(t_n))\right)   \right\}=y_i(t_n)\zeta_i^mR_i^{n,m}, \qquad  \begin{array}{l}
             i=1,\dots,N,  \\
             m=0,\dots,M 
        \end{array}
    \end{split}
    \end{equation*}
    and therefore, recalling $\tau_m\in[0,1],$ we conclude
    \begin{equation*}
        \max_{m=0,\dots,M}\|\bm{\delta}^{n,m}_{(0)}(h)\|_\infty\leq \mathrm{C}_\mathrm{R}\beta \exp{(T\hat{\Phi}^{[0]}/N_t)} \, h^2 =\mathcal{O}(h^2).
    \end{equation*}
    Consider now $k>0$ and assume that, for each stage index $m,$ $\|\bm{\delta}^{n,m}_{(k-1)}(h)\|_\infty=\mathcal{O}(h^{M+2})+\mathcal{O}(h^{k+1}),$ that is
     \begin{equation}\label{eq:Induction_Hp}
        \max_{\substack{i=1,\dots,N \\ m=0,\dots,M}}|\delta^{n,m}_{i(k-1)}(h)|\leq C_1h^{M+2}+C_2h^{k+1}, \qquad C_1,C_2\in \mathbb{R}^+.
    \end{equation}
    For sufficiently small stepsizes, given the induction hypothesis, $\bm{y}_{(k-1)}(t)\in \Omega$ for $t\in[0,T].$ From the mean value theorem and the integral formulation of the weights in \eqref{eq:forma_pesi}, there exist some constants $\xi_{i,m}\in(0,e^{h\hat{\Phi}_i^{[0]}})$ such that
    \begin{equation}\label{eq:tmp1_local_error}
        \begin{split}
            \delta^{n,m}_{i(k)}(h)&=y_i(t_n)\left\{ \exp\left(\int_{t_n}^{t_{n,m}}\varphi_i(\bm{y}(s)) \ ds\right)-\exp\left(h\sum_{j=0}^MQ_{m,j}\varphi_i(\bm{y}_{(k-1)}(t_{n,j}))\right)\right\} \\
            &=y_i(t_n)\xi_{i,m}\left\{ h\bigintss_{ \, 0}^{\tau_m} \!\!\! \left(\varphi_i(\bm{y}(t_n+sh)) -\sum_{j=0}^M\ell_j(s)\varphi_i(\bm{y}(t_{n,j})) \right)  ds  +h\sum_{j=0}^MQ_{m,j}\left(\varphi_i(\bm{y}(t_{n,j}))- \varphi_i(\bm{y}_{(k-1)}(t_{n,j}))\right)\right\}.
        \end{split}
    \end{equation}
    The $k$--th sweep local error in \eqref{eq:tmp1_local_error} arises from two distinct contributions: a quadrature error and a defect-iteration error, which are analysed in the following. Well-established results on the interpolation remainder (see, for instance \cite[Theorem 8.2]{Quarteroni_2007} and references therein) yield
    \begin{equation*}
        Err_{\mathcal{E}^{-1}_M\bm{\varphi}_i}(s):=\varphi_i(\bm y(t_n+sh))
        -
        \sum_{j=0}^M \ell_j(s)\,
        \varphi_i(\bm y(t_{n,j}))
        =\left.\dfrac{d^{M+1}\varphi_i(\bm y(t_n+\sigma h))}{d\sigma^{M+1}}
        \right|_{\sigma=\theta(s)} \
        \prod_{j=0}^M \dfrac{(s-\tau_j)}{(M+1)!}, \quad \text{with} \quad \begin{array}{l}
             s\in[0,1] \ \text{and}  \\
             \theta(s)\in(0,1),
        \end{array}
    \end{equation*}
    which, denoted by $\hat{F}=\max_{i=1,\dots,N} \|f_i\|_{L^\infty(\Omega)},$ implies
    \begin{equation}\label{eq:err_interp_integr}
        \begin{split}
           \max_{\substack{i=1,\dots,N \\ m=0,\dots,M}}\left|h\int_0^{\tau_m} Err_{\mathcal{E}^{-1}_M\bm{\varphi}_i}(s) \ ds\right|\leq h \left(h^{M+1} \  \dfrac{\hat{\Phi}^{[M+1]} \hat{F}^{M+1}}{(M+1)!}\max_{m=0,\dots,M}\prod_{j=0}^{M}\int_0^{\tau_m}|s-\tau_j| \ ds \right) \leq \mathrm{C}_{\mathcal{E}^{-1}_M} h^{M+2}=\mathcal{O}(h^{M+2}),
        \end{split}
    \end{equation}
    with $\mathrm{C}_{\mathcal{E}^{-1}_M}$ positive constant independent of the stepsize $h$ and of the step index $n.$
    Furthermore, recalling the induction hypothesis \eqref{eq:Induction_Hp}, we derive
    \begin{equation}\label{eq:lip_err}
        h\max_{\substack{i=1,\dots,N \\ m=0,\dots,M}}\left|\sum_{j=0}^MQ_{m,j}\left(\varphi_i(\bm{y}(t_{n,j}))- \varphi_i(\bm{y}_{(k-1)}(t_{n,j}))\right)\right|\leq h \hat{\Phi}^{[1]} \sum_{j=0}^M|Q_{m,j}| \max_{\substack{i=1,\dots,N \\ m=0,\dots,M}}|\delta^{n,m}_{i(k-1)}|\leq \tilde{C}_1\hat{\Phi}^{[1]}h^{M+3}+\tilde{C}_2\hat{\Phi}^{[1]}h^{k+2}.
    \end{equation}
    Therefore, by gathering \eqref{eq:lip_err} and \eqref{eq:err_interp_integr} with \eqref{eq:tmp1_local_error} we get 
    \begin{equation*}
        \max_{\substack{i=1,\dots,N \\ m=0,\dots,M}}|\delta^{n,m}_{i(k)}(h)|\leq \beta \exp{(T\hat{\Phi}^{[0]}/N_t)} \left( (\mathrm{C}_{\mathcal{E}^{-1}_M}+h\tilde{C}_1\hat{\Phi}^{[1]})h^{M+2}+\tilde{C}_2\hat{\Phi}^{[1]}h^{k+2} \right),
    \end{equation*}
    that leads to \eqref{eq:Induction_Objective}. This shows that each correction sweep increases the accuracy by one order, until the order of the underlying quadrature rule is reached. \\ Finally, from the last stage of the last sweep, the SPIDeC local truncation error \eqref{eq:Local_error}$_{3}$ satisfies
    \begin{equation*}
        \max_{i=1,\dots,N} |\delta_i(h;t_{n+1})|=\mathcal{O}(h^{M+2})+\mathcal{O}(h^{N_k+2})=\mathcal{O}(h^{p+1}),
    \end{equation*}
    which proves the $p$-th order consistency of the SPIDeC method.
\end{proof}

To investigate the convergence properties of the SPIDeC integrator \eqref{eq:SPIDeC}, we first establish some
preparatory results. The multiplicative–exponential structure of the scheme implies that each stage is a small relative perturbation of the previous iterate. By compactness and interior inclusion, this guarantees the existence of a safety margin that prevents the numerical solution from leaving an admissible compact set for sufficiently small time steps. This is proved with the following result.

\begin{lemma}
    \label{lem:SPIDeC_invariant_set}
    Let $\{\bm{y}^{n}\}$ and $\{\bm{y}_{(k)}^{n,m}\}$ denote the solutions to the discrete equations \eqref{eq:SPIDeC}, with $\bm{y}^{0}=\bm{y}(0)$. Assume that $\bm f \in C^{0}(\Omega)$, where $\Omega \subset \mathbb{R}_{+}^{N}$ is a compact set satisfying \eqref{eq:def_Omega}. Then, there exists a positive threshold $h^*$ such that, for all $0<h\leq h^*$, the SPIDeC solution satisfies
        \begin{equation*}
            \bm y^n \in \Omega, \qquad \bm y^{n,m}_{(k)} \in \Omega, \qquad  \text{for all } \ n=0,\dots,N_t, \qquad m=0,\dots,M \qquad \text{ and } \qquad k=0,\dots,N_k.
        \end{equation*}
\end{lemma}
    
\begin{proof}
    By construction of $\Omega$, there exists $\delta\in(0,\alpha/2)$ such that
    \begin{equation*}
        \Omega_{\delta}=\{\bm{\omega}:\alpha-\delta\leq \omega_i \leq \beta+\delta\}\subset\Omega_{2\delta}=\{\bm{\omega}\in\mathbb{R}_{+}^{N}:\alpha-2\delta\leq \omega_i \leq \beta+2\delta\}\subset \Omega.
    \end{equation*}
    
    Since $\bm f$ is continuous on $\Omega$, it is bounded on $\Omega_{2\delta}$. Hence
    \begin{equation*}
        \max_{\bm{\omega}\in\Omega_{2\delta}}\left| \frac{f_i(\bm\omega)}{\omega_i} \right|\leq \Phi_\delta = \max_{i=1,\dots,N}\frac{\|f_i\|_{L^\infty(\Omega)}}{\alpha-2\delta}, \qquad \max_{\bm{\omega}\in\Omega_{2\delta}}\left| \sum_{j=0}^M Q_{m,j}\frac{f_i(\bm{\omega})}{\omega_i} \right|\leq C_Q \Phi_\delta,\qquad C_Q = \max_{m=0,\dots,M}\sum_{j=0}^M |Q_{m,j}|\geq 1,
    \end{equation*}
    uniformly with respect to $m=0,\dots,M.$ Furthermore, denoted by
    \begin{equation*}
         h^* =\dfrac{1}{C_Q \Phi_\delta}\min\left\{\log\!\left(\dfrac{\alpha-\delta}{\alpha-2\delta}\right),\;\log\!\left(\dfrac{\beta+2\delta}{\beta+\delta}\right)\right\},
    \end{equation*}
    the inequalities $\exp(-h\Phi_\delta)\geq (\alpha-2 \delta)/(\alpha-\delta)$ and $\exp(h\Phi_\delta)\leq (\beta+2\delta)/(\beta+\delta)$ hold true for each $h\in(0,h^*].$ 
    We preliminarily observe that, by construction and recalling \eqref{eq:def_Omega}, we have $\bm y^0 = \bm y(0) \in \Omega_\delta.$ We now prove by induction on $n \ge 0$ that, if $h < h^*$, then $\bm y^{n} \in \Omega_\delta$ implies $\bm y^{n+1} \in \Omega.$ From the equation \eqref{eq:SPIDeC}$_1$, it follows that
    \begin{equation*}
        \alpha -2\delta \leq (\alpha-\delta) e^{-h\Phi_\delta} \leq y_i^n e^{-h\Phi_\delta} \leq y^{n,m}_{i(0)} \leq y_i^n e^{h\Phi_\delta}\leq (\beta+\delta) e^{h\Phi_\delta} \leq \beta +2\delta, \qquad i=1,\dots,N, \qquad m=0,\dots,M,
    \end{equation*}
    and therefore each predictor stage vector $\bm{y}_{(0)}^{n,m}$ belongs to $\Omega_{2\delta}.$ Consider $k>0$ and assume that $\bm y^{n,j}_{(k-1)} \in \Omega_{2\delta}$ for all $j=0,\dots,M$. Then, independently of $0<h\leq h^*$, the correction step \eqref{eq:SPIDeC}$_2$ yields
    \begin{equation*}
         \alpha -2\delta \leq(\alpha-\delta) e^{-h C_Q \Phi_\delta}\leq y_i^n e^{-h C_Q \Phi_\delta} \leq y^{n,m}_{i(k)} \leq y_i^n e^{h C_Q \Phi_\delta} \leq (\beta+\delta) e^{h C_Q \Phi_\delta}\leq \beta +2\delta, \qquad i=1,\dots,N, \qquad m=0,\dots,M.
    \end{equation*}
    As a consequence, for $0<h\leq h^*$, all  $\bm{y}^{n,m}_{(k)}$ belong to $\Omega_{2\delta}$ for every $k=0,\dots,N_k$ and $m=0,\dots,M$. Since $\bm y^{n+1} = \bm y^{n,M}_{(N_k)}$, it follows that $\bm y^{n+1} \in \Omega_{2\delta}\subset \Omega,$ which yields the result.
\end{proof}

\begin{lemma}
\label{lem:component-wise_lipschitz}
Let $\Omega \subset \mathbb{R}_{+}^{N}$ be a compact set and assume that $\varphi_i :\Omega\to \mathbb{R}$ is continuously differentiable for each $i=1,\dots,N$. Given $h\in \mathbb{R}^+$ and $q_0,\dots,q_M \in \mathbb{R}$ define, for $\bm{x}\in\Omega$ and $X = [\bm{x}^{(1)},\bm{x}^{(2)},\dots,\bm{x}^{(M+2)}] \in \Omega^{M+2},$
\begin{equation}\label{eq:Psi_lem}
    \psi_i(\bm{x}) = x_i \exp\!\big(h\,\varphi_i(\bm{x})\big) \qquad \text{and} \qquad \Psi_i(X) = x^{(1)}_i \exp\!\left(h\sum_{j=0}^M q_j \varphi_i\big(\bm{x}^{(j+2)}\big)\right), \qquad i=1,\dots,N.
\end{equation}
Then, for each $i=1,\dots,N$, the functions $\psi_i$ and $\Psi_i$ are Lipschitz continuous on $\Omega$ and $\Omega^{M+2}$, respectively.
\end{lemma}

\begin{proof}
Since $\Omega$ is compact and $\varphi_i \in C^1(\Omega)$, it follows that $\varphi_i$ is Lipschitz continuous and bounded on $\Omega$. Hence, there exists $L_i > 0$ such that $
|\varphi_i(\bm{x}) - \varphi_i(\bm{y})| \leq L_i \|\bm{x} - \bm{y}\|,\  \forall \bm{x},\bm{y} \in \Omega.$  A straightforward application of the mean value theorem to the exponential then yields
\begin{equation*}
    \begin{aligned}
        |\psi_i(\bm{x}) - \psi_i(\bm{y})|
        &= \left| x_i e^{h \varphi_i(\bm{x})} - y_i e^{h \varphi_i(\bm{y})} \right| \leq |x_i - y_i| e^{h \|\varphi_i\|_{\infty}}
        + |y_i| \left| e^{h \varphi_i(\bm{x})} - e^{h \varphi_i(\bm{y})} \right| \leq L_\psi \|\bm{x} - \bm{y}\|, \quad \forall \bm{x},\bm{y} \in \Omega, 
    \end{aligned}
\end{equation*}
where $L_\psi=\max_{j=1,\dots,N}\left( 1+ h R L_j    \right) e^{h \|\varphi_j\|_{\infty}}=1+\mathcal{O}(h),$  $R=\max_{\bm{\omega}\in\Omega}\|\bm{\omega}\|$ and $\|\varphi_j\|_{\infty}=\max_{\bm{\omega}\in\Omega}|\varphi_j(\bm{\omega})|.$ \\
Analogously, for all $X,Y \in \Omega^{M+2},$
\begin{equation*}
    \left|h\sum_{j=0}^M q_j \varphi_i\left(\bm{x}^{(j+2)}\right)-h\sum_{j=0}^M q_j \varphi_i\left(\bm{y}^{(j+2)}\right) \right|\leq h L_i \Theta \|X-Y\|, \qquad i=1,\dots,N,
\end{equation*}
where $\Theta=\sum_{j=0}^M|q_j|.$ Therefore, proceeding as in the first part of the proof, we have
\begin{equation*}
    |\Psi_i(X) - \Psi_i(Y)| \leq |x_i^{(1)} - y_i^{(1)}| e^{h \Theta \|\varphi_i\|_\infty}
+ |y_i^{(1)}| \left| e^{h\sum_{j=0}^M q_j \varphi_i\left(\bm{x}^{(j+2)}\right)} - e^{h\sum_{j=0}^M q_j \varphi_i\left(\bm{y}^{(j+2)}\right)} \right|\leq L_\Psi \|X-Y\|,
\end{equation*}
for any matrix norm compatible with the Euclidean norm, where $L_\Psi=\max_{j=1,\dots,N}\left( 1+ h \Theta R L_j    \right) e^{h \|\varphi_j\|_{\infty}}=1+\mathcal{O}(h),$  $R=\max_{\bm{\omega}\in\Omega}\|\bm{\omega}\|$ and $\|\varphi_j\|_{\infty}=\max_{\bm{\omega}\in\Omega}|\varphi_j(\bm{\omega})|.$
\end{proof}

The iterative structure of the SPIDeC method \eqref{eq:SPIDeC} leads to a natural decomposition of the global discretization error into stage-wise and step-update contributions, defined as
\begin{equation*}
    \begin{aligned}
        &\bm{e}_{(k)}^{n,m}(h)=[e_{1(k)}^{n,m}(h),\dots,e_{N(k)}^{n,m}(h)]^\top=\bm{y}(t_{n,m})-\bm{y}_{(k)}^{n,m}, \qquad &&m=0,\dots,M,  \\
        &\bm{e}(h;t_n)=\bm{e}_{(N_k)}^{n,M}(h), \qquad &&n=0,\dots,N_t.
    \end{aligned}
\end{equation*}
The following theorem establishes the high order convergence of the unconditional positive approximations of the class \eqref{eq:SPIDeC}.

\begin{theorem}[Convergence of the SPIDeC scheme]\label{thm:Convergence_SPIDeC}
Let $\bm{y}(t)$ be the continuous solution of \eqref{eq:ODE_system} for $t \in [0,T]$ under the positivity Assumption \ref{ass:Positivity}. For a given $N_t \in \mathbb{N}$, let $\{\bm{y}^n\}_{n=0}^{N_t}$ denote the numerical approximation obtained by the SPIDeC scheme \eqref{eq:SPIDeC} with $M+1$ stages, $N_k$ correction sweeps and a uniform stepsize $h = T/N_t$. Assume that $\bm{f} \in C^q(\Omega)$, where $q = \max\{M+1, N_k+1\}$ and $\Omega \subset \mathbb{R}^N_+$ is a compact set satisfying \eqref{eq:def_Omega}.
Then, the global approximation error satisfies
\begin{equation*}
    \max_{0 \leq n \leq N_t} \|\bm{e}(h; t_n)\| = \max_{0 \leq n \leq N_t} \|\bm{y}(t_n) - \bm{y}^n\| = \mathcal{O}(h^p),
\end{equation*}
where $p = \min\{M+1, N_k+1\}$. Consequently, the SPIDeC method \eqref{eq:SPIDeC} is convergent of order $p$.
\end{theorem}
\begin{proof}
We first observe that, by Lemma \ref{lem:SPIDeC_invariant_set}, the stepsize $h$ can be chosen sufficiently small so that both the continuous and numerical solutions remain in $\Omega$. 
Let us introduce the auxiliary functions $\varphi_i \in C^q(\Omega)$, for $i=1,\dots,N$, defined by $\varphi_i(\bm{\omega})=f_i(\bm{\omega})/\omega_i.$ 
Following the notation in \eqref{eq:Psi_lem}, we define the mappings $\psi_i^m:\Omega\to\mathbb{R}^+$ and $\Psi_i^m:\Omega^{M+2}\to\mathbb{R}^+$ as
\begin{equation*}
    \psi_i^m(\bm{x}) = x_i \exp\!\big(h\tau_m\,\varphi_i(\bm{x})\big) \qquad \text{and} \qquad 
    \Psi_i^m(X) = x^{(1)}_i \exp\!\left(h\sum_{j=0}^M Q_{m,j} \varphi_i\big(\bm{x}^{(j+2)}\big)\right).
\end{equation*}
By construction of the predictor step, the global discretization error satisfies
\begin{equation*}
    e_{i(0)}^{n,m}(h)=y_i(t_{n,m})-y_{i(0)}^{n,m}
    =\delta_{i(0)}^{n,m}(h)+(\psi_i^m(\bm{y}(t_n))-\psi_i^m(\bm{y}^n)), \qquad m=0,\dots,M.
\end{equation*}
From the Lipschitz continuity of Lemma \ref{lem:component-wise_lipschitz}, it follows that
\begin{equation}\label{eq:Err_PRedictor}
    \max_{m=0,\dots,M}\|\bm{e}_{(0)}^{n,m}(h)\|_\infty
    \leq \max_{m=0,\dots,M} \|\bm{\delta}_{(0)}^{n,m}(h)\|_\infty
    +(1+C_{0}h) \, \|\bm{e}(h;t_n)\|_\infty.
\end{equation}
An analogous argument applies to the correction sweeps. For each correction index $k=1,\dots,N_k$, we have
\begin{equation*}
    e_{i(k)}^{n,m}(h)
    =\delta_{i(k)}^{n,m}(h)
    +(\Psi_i^m(Y_{(k)}(t_n))-\Psi_i^m(Y_{(k)}^n)), \qquad m=0,\dots,M,
\end{equation*}
where $Y_{(k)}(t_n)=\left[\bm{y}(t_n),\bm{y}_{(k-1)}(t_{n,0}),\cdots,\bm{y}_{(k-1)}(t_{n,M})\right]$ and $
Y_{(k)}^n=\left[\bm{y}^n,\bm{y}_{(k-1)}^{n,0},\cdots,\bm{y}_{(k-1)}^{n,M}\right].$
Using again Lemma \ref{lem:component-wise_lipschitz}, together with the definition of $Y_{(k)}$, we obtain
\begin{equation}\label{eq:tmp_clean}
\begin{split}
\max_{m=0,\dots,M}\|\bm{e}_{(k)}^{n,m}(h)\|_\infty
&\leq \max_{m=0,\dots,M} \|\bm{\delta}_{(k)}^{n,m}(h)\|_\infty
+(1+C_{k}h)\left(\max_{m=0,\dots,M}\|\bm{e}_{(k-1)}^{n,m}(h)\|_\infty+\|\bm{e}(h;t_n)\|_\infty\right).
\end{split}
\end{equation}
Iterating \eqref{eq:tmp_clean} over $k$ and exploiting \eqref{eq:Err_PRedictor} yields
\begin{equation*}
\max_{m=0,\dots,M}\|\bm{e}_{(k)}^{n,m}(h)\|_\infty
\leq \sum_{j=0}^{k}\left(\prod_{i=j+1}^{k} (1+C_i h)\right)\max_{m=0,\dots,M} \|\bm{\delta}_{(j)}^{n,m}(h)\|_\infty
+\left(\prod_{i=0}^{k} (1+C_i h)\right)\|\bm{e}(h;t_n)\|_\infty, \qquad k=0,\dots,N_k.
\end{equation*}
By consistency (cf. Theorem \ref{thm:Consistency}), $\|\bm{\delta}_{(k)}^{n,m}(h)\|_\infty\leq C_\delta h^{p+1}$. 
Since the constants $C_k$ are independent of $h$, there exists $\hat{C}=\max_{k} C_k>0$ such that $\prod_{i=j}^{k}(1+C_i h)\leq (1+\hat{C}h)^{k-j+1}.$ Applying the previous estimate with $k=N_k$ and $m=M$, we obtain the recursive bound
\begin{equation*}
    \|\bm{e}(h;t_{n+1})\|_\infty\leq (1+\hat{C}h)\|\bm{e}(h;t_n)\|_\infty + \tilde{C}_\delta h^{p+1},
\end{equation*}
for a suitable constant $\tilde{C}_\delta>0$ independent of $h$.
A standard geometric sum estimate then yields
\begin{equation*}
    \|\bm{e}(h;t_{n+1})\|_\infty\leq \tilde{C}_\delta h^{p+1} \sum_{j=0}^{n} (1+\hat{C}h)^{n-j}  \leq \tilde{C}_\delta h^{p+1} \dfrac{\big((1+\hat{C}h)^{n+1}-1\big)}{\hat{C}h}\leq \dfrac{\tilde C_\delta}{\hat{C}}\big(\exp(\hat{C}T)-1\big)\,h^p.
\end{equation*}
Taking the maximum over $n=0,\dots,N_t$ completes the proof.
\end{proof}

The findings of this section show that the SPIDeC framework ensures unconditional positivity and arbitrarily high order convergence, while preserving the qualitative structure of the underlying system, and overcomes the classical order barriers for linear numerical integrators~\cite{Bolley_1978}. This relies on a non-linear multiplicative--exponential formulation, which guarantees positivity for any step size $h>0$, and on the deferred correction mechanism, which can be iterated to attain any prescribed order $p$.

\section{Stability Analysis}\label{sec:Stability}
In this section, we investigate the 
stability properties of the SPIDeC scheme~\eqref{eq:SPIDeC}. Firstly, we apply it to the classical Dahlquist test 
equation \cite[Chapter IV.2]{Hairer_1996} and show that it reproduces 
the exact solution on linear problems, independently of the 
discretization parameters $h$, $M$, and $N_k.$ We then extend the analysis to the non-linear setting, 
establishing contractivity properties for dissipative systems.

\subsection{Linear stability and exact semigroup preservation}\label{subsec:LinStab} 

Here, we focus on the scalar Dahlquist test problem
\begin{equation}\label{eq:Dal_Problem}
    y'(t)=\lambda\, y(t),
    \qquad y(0)=1,
    \qquad \lambda\in\mathbb{C},
    \quad \operatorname{Re}(\lambda)<0,
\end{equation}
whose solution is $y(t)=\exp(t \lambda)$, which constitutes the benchmark for linear stability analysis, as it describes the dynamics of a single eigenmode of a linear operator. The following result establishes the exact integration of linear problems and the $L$-stability (cf. \cite[Definition 3.7]{Hairer_1996}) of the SPIDeC scheme \eqref{eq:SPIDeC}.
\begin{theorem}
\label{Thm:Linear_Stability}
    Let $y(t)$ and $\{y_\lambda^n\}_{n\geq 0}$ denote the continuous and SPIDeC numerical solution to the Dahlquist equation \eqref{eq:Dal_Problem}, respectively. Then, for any $h>0$, $M\geq 0$ and $N_k\geq 0$,
    \begin{equation*}
        y_\lambda^n = \exp(nh \ \lambda)=y(t_n),
        \qquad n\geq 0.
    \end{equation*}
    Consequently, the stability function of the SPIDeC 
    scheme~\eqref{eq:SPIDeC} coincides with that of the continuous flow
    \begin{equation*}
        R(z) = \exp(z), \qquad z = h\lambda.
    \end{equation*}
    In particular, $|R(z)|<1$ for all $\operatorname{Re}(z)<0$, 
    so that the method is $A$-stable. Moreover,
    \[
        \lim_{\operatorname{Re}(z)\to-\infty} |R(z)| = 0,
    \]
    so that the method is $L$-stable.
\end{theorem}
\begin{proof}
    For the scalar test equation~\eqref{eq:Dal_Problem}, $f(y)=\lambda y$, 
    so $f(y)/y = \lambda$ for all $y\neq 0$. The predictor 
    stage of the SPIDeC scheme gives
    \begin{equation*}
         y_{\lambda(0)}^{n,m}
        = y_\lambda^n\exp(h\,\tau_m\,\lambda),
        \qquad m=0,\dots,M.
    \end{equation*}
    For any correction sweep $k\geq 1$, using $f(y)/y=\lambda$ 
    and the consistency identity 
    $\sum_{j=0}^M Q_{m,j}=\tau_m$ 
    (cf.\ Section~\ref{subsec:Quadrature}),
    \[
        y_{\lambda(k)}^{n,m}
        = y_\lambda^n\exp\!\left(
            h\sum_{j=0}^M Q_{m,j}\,\lambda
          \right)
        = y_\lambda^n\exp(h\,\tau_m\,\lambda),
        \qquad m=0,\dots,M.
    \]
    Hence all stages, at every correction level, assume the 
    same value, 
    independently of $h$. Taking $m=M$ and recalling $\tau_M=1$,
    \[
        y_\lambda^{n+1}
        = y_\lambda^n\,\exp(h \lambda)
        = \exp((n+1)h \ \lambda)
        = y(t_{n+1}),
    \]
    which establishes $R(z)=e^z$ and the exact reproduction property 
    by induction on $n$. The $A$- and $L$-stability claims follow 
    immediately from $|e^z|=e^{\operatorname{Re}(z)}$.
\end{proof}
The exact reproduction property of Theorem~\ref{Thm:Linear_Stability} extends component-wise to diagonal linear systems
\begin{equation}\label{eq:Diag_Lin_Sys}
    \bm{y}'(t)= \bm{\Lambda} \bm{y}(t), \quad \bm{y}(0)=[1,\dots,1]^\top,  \quad \bm{\Lambda} = \mathrm{diag}([\lambda_1,\dots,\lambda_N])\in\mathbb{R}^{N\times N}, \quad \operatorname{Re}(\lambda_i)\leq 0, \quad  i=1,\dots,N.
\end{equation}
In particular, for any stepsize $h$, correction index $k$, and stage $m$, we have $\bm{y}_{(k)}^{n,m} = e^{\bm{\Lambda} t_{n,m}} \bm{y}^0 = \bm{y}(t_{n,m}).$ By Theorem~\ref{Thm:Linear_Stability}, the SPIDeC discretization \eqref{eq:SPIDeC} preserves exactly the continuous semigroup $e^{t\bm{\Lambda}}\bm{y}^0$ generated by diagonal linear operators. In particular, the discrete evolution coincides with the exact flow at all grid points, so that no qualitative distortion of the linear dynamics is introduced by the discretization, independently of $h$, $M$ and $N_k$.

\begin{remark}
    $A$-stable Runge--Kutta methods admit a rational stability function $R(z)=P(z)/Q(z)$ approximating the exponential. 
    For $z\in\mathbb{R}^-$, such approximations may lose monotonicity, since $R(z)$ can change sign and induce spurious oscillations, in contrast with the strictly decaying continuous solution. 
    This is a consequence of a classical result in rational approximation theory, which implies that any rational approximation of $e^z$ on $\mathbb{R}^-$ must eventually become negative~\cite{Iserles_2009}.
\end{remark}

\subsection{Non-linear stability and logarithmic contractivity}
\label{subsec:NonlinStab}

The linear stability analysis of Section~\ref{subsec:LinStab} exploits the structure of the Dahlquist equation which decouples the correction mechanism entirely. For general non-linear systems, the stability properties of SPIDeC are naturally analyzed by working in logarithmic variables, for which the collocation scheme reduces to an implicit Runge--Kutta collocation method. This identification enables a systematic non-linear stability theory based on the classical framework of B-stable methods, for which we refer the reader to \cite[Chapter IV.12]{Hairer_1996} and \cite[Chapter 3, Section 35]{Butcher}. 

Given the positivity Assumption \eqref{ass:Positivity}, the component-wise change of variables $u_i = \log(y_i),$ $i=1,\dots,N,$ transforms the original system~\eqref{eq:ODE_system} into
\begin{equation}\label{eq:LogSystem}
    u^{\prime}_i(t) = G_i(\bm{u}(t)),
    \qquad \text{with} \qquad
    G_i(\bm{u}) = \dfrac{f_i(\operatorname{Exp}(\bm{u}))}{\exp(u_i)},
    \qquad i=1,\dots,N,
\end{equation}
where $\bm{u}=[u_1,\dots,u_N]^\top$ and $\operatorname{Exp}(\bm{u}) = [\exp(u_1),\dots,\exp(u_N)]^\top$. The following result shows that the SPIDeC collocation map $\bm{\Psi}^h$ defined in \eqref{eq:CollocationMap} corresponds to the implicit Runge--Kutta (RK) collocation method with tableau
\begin{equation}\label{eq:Tableau}
    \begin{array}{c|cccc}
    \tau_0 & Q_{0,0} & Q_{0,1} & \cdots & Q_{0,M} \\
    \tau_1 & Q_{1,0} & Q_{1,1} & \cdots & Q_{1,M} \\
    \vdots & \vdots  & \vdots  & \ddots & \vdots  \\
    \tau_M & Q_{M,0} & Q_{M,1} & \cdots & Q_{M,M} \\
    \hline
           & Q_{M,0} & Q_{M,1} & \cdots & Q_{M,M}
    \end{array}
\end{equation}
applied to the logarithmic system \eqref{eq:LogSystem}.

\begin{theorem}\label{prop:RK_Log}
    Under the Assumption~\ref{ass:Solv}, let $\tilde{\bm{Y}}^n=[\tilde{\bm{y}}^{n,0},\dots,\tilde{\bm{y}}^{n,M}]$ denote the unique collocation solution to \eqref{eq:coll_syst_Con_L}. Then, the logarithmic formulation $\tilde{u}_i^{n,m} = \log(\tilde{y}_i^{n,m})$ satisfies the implicit Runge--Kutta system
    \begin{equation}\label{eq:RK_Log}
        \tilde{u}_i^{n,m}=  u_i^n + h\sum_{j=0}^M Q_{m,j}\,G_i\!\left(\tilde{\bm{u}}^{n,j}\right), \quad \begin{array}{l}
             i=1,\dots,N,  \\
             m=0,\dots,M, 
        \end{array} \quad \text{with output} \quad u_i^{n+1} = u_i^n + h\sum_{j=0}^M b_j\, G_i\!\left(\tilde{\bm{u}}^{n,j}\right), \quad b_j := Q_{M,j}.
    \end{equation}
\end{theorem}
\begin{proof}
    The collocation system~\eqref{eq:coll_syst_Con_L} reads
    \begin{equation*}
        \log(\tilde{y}_i^{n,m})= \log(y_i^n)+ h\sum_{j=0}^M Q_{m,j}\frac{f_i(\tilde{\bm{y}}^{n,j})}{\tilde{y}_i^{n,j}},
        \qquad \begin{array}{l}
             i=1,\dots,N,  \\
             m=0,\dots,M. 
        \end{array}
    \end{equation*}
    Recalling the change of variables and the definition of $G_i$ in \eqref{eq:LogSystem}, this is exactly~\eqref{eq:RK_Log}. The output formula then follows by taking $m=M.$
\end{proof}

The non-linear stability of the RK method~\eqref{eq:RK_Log} is governed by the dissipativity of the log-transformed vector field $\bm{G}$. We introduce the relevant condition as follows.
\begin{definition}
\label{def:LogDiss}
    The system~\eqref{eq:ODE_system} is called 
    \emph{logarithmically dissipative} with 
    constant $\nu\in\mathbb{R}$ if the function $\bm{G}=[G_1,\dots,G_N]^\top$, with $G_i$ defined in~\eqref{eq:LogSystem}, 
    satisfies the one-sided Lipschitz condition
    \begin{equation}\label{eq:OSLC}
        \langle
          \bm{G}(\bm{u}) - \bm{G}(\bm{v}),\,
          \bm{u} - \bm{v}
        \rangle
        \leq
        \nu\,\|\bm{u}-\bm{v}\|^2,
        \qquad
        \forall\,\bm{u},\bm{v}\in\mathbb{R}^N.
    \end{equation}
    When $\nu<0$, the system is called 
    \emph{strictly logarithmically dissipative}.
\end{definition}
If the function $\bm{G}$ is continuously differentiable on $\mathbb{R}_{+}^{N}$, the condition \eqref{eq:OSLC} is equivalent to $\mu(\bm{J}_{\bm{G}})\leq \nu,$ (see \cite[Chapter I.10]{Hairer_I}) where $\mu$ denotes the logarithmic norm of matrices and $\bm{J}_{\bm{G}}$ is the Jacobian of $\bm{G}.$ The logarithmic dissipativity condition \eqref{eq:OSLC} holds for several classes of interest, such as in case of Lotka--Volterra systems ($G_i(\bm{u}) = r_i + \sum_{j=1}^{n}A_{ij}\exp(u_j)$) with negative semidefinite interaction matrix $A$. 

We consider the \emph{log-Euclidean metric} on the positive $N$-dimensional cone defined as
\begin{equation}\label{eq:distance}
    \operatorname{d}(\bm{y},\bm{z})= \|\operatorname{Log}(y\oslash\bm{z})\|= \|\operatorname{Log}(\bm{y})-\operatorname{Log}(\bm{z})\|, \qquad \text{where} \qquad \operatorname{Log}(\bm{v})=[\log(v_1),\dots,\log(v_N)]^\top,
\end{equation}
which, when restricted to any compact subset of $\mathbb{R}_{+}^{N}$, is equivalent to the standard Euclidean norm. The function $\operatorname{d}:\mathbb{R}_{+}^{N}\times\mathbb{R}_{+}^{N}\to \mathbb{R}_{\geq 0}$ represents a natural distance for positive dynamical systems, since it is invariant under component-wise rescaling $\bm{y}\mapsto D\bm{y}$ for any diagonal $D>0$, and it metrizes the group structure of $(\mathbb{R}_{+}^{N}, \cdot)$. In case of strictly logarithmically dissipative systems, by the standard Gronwall inequality applied to~\eqref{eq:LogSystem}, two exact solutions $\bm{y}(t)$ and $\bm{z}(t)$ of~\eqref{eq:ODE_system} satisfy
\begin{equation}\label{eq:ContContractivity}
     \operatorname{d}(\bm{y}(t),\bm{z}(t))
        \leq \exp(\nu t)\,\operatorname{d}(\bm{y}(0),\bm{z}(0))\leq \operatorname{d}(\bm{y}(0),\bm{z}(0)),
        \qquad t\geq 0.
\end{equation}

Here, we are interested in the unconditional preservation, at the level of the collocation map, of the contractivity property \eqref{eq:ContContractivity}. The following result shows that the SPIDeC collocation map inherits the  algebraic stability (cf. \cite[Definition 357B]{Butcher}) of the underlying Runge--Kutta method in logarithmic variables \eqref{eq:RK_Log}. 
\begin{theorem}[Contractivity of the collocation map]
\label{thm:NonlinContr}
    Let Assumption~\ref{ass:Solv} hold and suppose that 
    the system~\eqref{eq:ODE_system} is logarithmically 
    dissipative with constant $\nu\leq 0$. Assume that 
    the RK scheme~\eqref{eq:RK_Log} is algebraically 
    stable, that is
    \begin{equation}\label{eq:Algebraic_Stability}
        b_j = Q_{M,j} \geq 0
        \quad \text{for each } j=0,\dots,M,
        \qquad \text{and} \qquad
        \bm{\mathcal{M}}
        = \operatorname{diag}(\bm{b})\,Q
          + Q^\top\operatorname{diag}(\bm{b})
          - \bm{b}\bm{b}^\top
        \ \text{is positive semi-definite,}
    \end{equation}
    where $\bm{b}=[b_0,\dots,b_M]^\top$.
    Then, for any $h>0$ and any 
    $\bm{y},\bm{z}\in\mathbb{R}_{+}^{N}$,
    \begin{equation}\label{eq:ContractColloc}
        \operatorname{d}\bigl(
          \bm{\Psi}^h(\bm{y}),\,
          \bm{\Psi}^h(\bm{z})
        \bigr)
        \leq \operatorname{d}(\bm{y},\bm{z}),
    \end{equation}
    where $\operatorname{d}$ denotes the log-Euclidean 
    metric~\eqref{eq:distance}. In particular, 
    the map $\bm{\Psi}^h$ in \eqref{eq:CollocationMap} is non-expansive for $\nu=0$ and 
    strictly contractive for $\nu<0$.
\end{theorem}

\begin{proof}
    By Theorem~\ref{prop:RK_Log}, the map $\bm{\Psi}^h$ 
    acts in log variables as the implicit RK 
    method~\eqref{eq:RK_Log} applied 
    to~\eqref{eq:LogSystem}. Let $\tilde{\bm{u}}^m$ and 
    $\tilde{\bm{v}}^m$ denote the stage values 
    corresponding to initial data 
    $\bm{u}^n=\operatorname{Log}(\bm{y})$ and 
    $\bm{v}^n=\operatorname{Log}(\bm{z})$, and define
    \[
        \bm{\theta} = \bm{u}^n - \bm{v}^n,
        \qquad
        \bm{\gamma}^m = \tilde{\bm{u}}^m 
                        - \tilde{\bm{v}}^m,
        \qquad
        \bm{\kappa}^m = \bm{G}(\tilde{\bm{u}}^m)
                        - \bm{G}(\tilde{\bm{v}}^m),
        \qquad m=0,\dots,M.
    \]
    Note that $\|\bm{\theta}\|
    = \operatorname{d}(\bm{y},\bm{z})$ and 
    $\|\bm{\gamma}^M\|
    = \operatorname{d}(\bm{\Psi}^h(\bm{y}),
                      \bm{\Psi}^h(\bm{z}))$.
    Subtracting the RK systems~\eqref{eq:RK_Log} for 
    the two initial data gives
    \[
        \bm{\gamma}^m 
        = \bm{\theta} 
          + h\sum_{j=0}^M Q_{m,j}\,\bm{\kappa}^j,
        \qquad m=0,\dots,M,
        \qquad\text{and}\qquad
        \bm{\gamma}^M 
        = \bm{\theta} 
          + h\sum_{j=0}^M b_j\,\bm{\kappa}^j.
    \]
    Expanding $\|\bm{\gamma}^M\|^2$ and substituting 
    $\bm{\theta} = \bm{\gamma}^j 
    - h\sum_{\ell=0}^M Q_{j,\ell}\bm{\kappa}^\ell$ yields
    \begin{equation}\label{eq:NormExpansion}
        \begin{split}
        \|\bm{\gamma}^M\|^2
        &= \|\bm{\theta}\|^2
           + 2h\sum_{j=0}^M b_j
             \langle\bm{\theta},\bm{\kappa}^j\rangle
           + h^2\Bigl\|\sum_{j=0}^M 
             b_j\bm{\kappa}^j\Bigr\|^2 \\
        &= \|\bm{\theta}\|^2
           + 2h\sum_{j=0}^M b_j
             \Bigl(\langle\bm{\gamma}^j,\bm{\kappa}^j\rangle
             - h\sum_{\ell=0}^M Q_{j,\ell}
               \langle\bm{\kappa}^\ell,
                       \bm{\kappa}^j\rangle\Bigr)
           + h^2\sum_{j=0}^M\sum_{\ell=0}^M
             b_j b_\ell
             \langle\bm{\kappa}^j,\bm{\kappa}^\ell
             \rangle \\
        &= \|\bm{\theta}\|^2
           + 2h\sum_{j=0}^M b_j
             \langle\bm{\gamma}^j,\bm{\kappa}^j\rangle
           - h^2\sum_{j=0}^M\sum_{\ell=0}^M
             \mathcal{M}_{j,\ell}
             \langle\bm{\kappa}^j,
                     \bm{\kappa}^\ell\rangle,
        \end{split}
    \end{equation}
    where $\mathcal{M}_{j,\ell} 
    = b_j Q_{j,\ell} + b_\ell Q_{\ell,j} - b_j b_\ell$.
    Since $\mathcal{M}\succcurlyeq 0$ by algebraic 
    stability, the last term is non-negative. The 
    one-sided Lipschitz condition~\eqref{eq:OSLC} gives
    \[
        \langle\bm{\gamma}^j,\bm{\kappa}^j\rangle
        = \langle\tilde{\bm{u}}^j - \tilde{\bm{v}}^j,\,
                  \bm{G}(\tilde{\bm{u}}^j)
                  - \bm{G}(\tilde{\bm{v}}^j)\rangle
        \leq \nu\|\bm{\gamma}^j\|^2 \leq 0.
    \]
    Since $\nu\leq 0$ and $b_j\geq 0$, the middle term 
    in~\eqref{eq:NormExpansion} is also non-positive. 
    Therefore $\|\bm{\gamma}^M\|^2 \leq \|\bm{\theta}\|^2$, 
    and taking square roots yields~\eqref{eq:ContractColloc}.
    
    If $\nu<0$, then $\langle\bm{\gamma}^j,\bm{\kappa}^j\rangle < 0$ whenever $\bm{\gamma}^j\neq\bm{0}$. Since $y \neq z$ implies $\theta \neq 0$, and Assumption \ref{ass:Solv} guarantees that the collocation 
solution depends continuously and injectively on 
the initial datum $\bm{u}^n$, we have 
$\gamma^j = \tilde{\bm{u}}^j - \tilde{\bm{v}}^j 
\neq \bm{0}$ for at least one $j\in\{0,\dots,M\}$, 
giving strict inequality.
\end{proof}

The algebraic stability condition~\eqref{eq:Algebraic_Stability} depends on the quadrature matrix $Q$, which is entirely determined by the choice of nodes $\{\tau_j\}_{j=0}^M$ via~\eqref{eq:forma_pesi}. Within the SPIDeC framework~\eqref{eq:SPIDeC}, the only structural constraint on the nodes is $\tau_M=1$, which is required by the time-marching update~\eqref{eq:SPIDeC}$_3$. The choice of $\tau_0$ is instead free, and different selections lead to different stability properties. The most natural choice satisfying $\tau_0=0$ is the Gauss--Lobatto family. However, the associated Lobatto~IIIA collocation tableau is algebraically stable only for $M=1$ (the implicit trapezoidal rule), and the condition~\eqref{eq:Algebraic_Stability} fails for $M\geq 2.$ By contrast, Gauss--Radau (right) nodes, characterized by $\tau_0>0$ and $\tau_M=1$, yield algebraically stable tableaux for all $M\geq 1$~\cite[Theorem~359C]{Butcher}, making Theorem~\ref{thm:NonlinContr} applicable unconditionally. Gauss--Legendre nodes also yield algebraically stable tableaux for all $M$~\cite[Chapter~IV.12]{Hairer_1996}, but require $\tau_M<1$, which conflicts with the time-marching structure of SPIDeC and cannot be accommodated. Therefore, within the SPIDeC framework, Gauss--Radau (right) nodes represent the canonical choice when unconditional non-linear contractivity is sought, while Gauss--Lobatto nodes remain a valid option for $M=1$ or when the constraint $\tau_0=0$ is desirable for other structural reasons. 

From Theorem \ref{thm:NonlinContr} the non-linear contractivity of the collocation map holds independently of the choice of $h.$ Conversely, as established with the following result, such property is recovered by the predictor scheme only for a sufficiently small stepsize.  

\begin{theorem}[Contractivity of the predictor for $\tau_0=0$]\label{thm:PredContr}
    Let $N_k=0$ and assume that the system~\eqref{eq:ODE_system}
    is logarithmically dissipative with constant $\nu\leq 0$.
    Suppose that $\tau_0=0$ and that $\bm{G}$ is Lipschitz continuous with constant $L.$ Then, for any $h>0$ and any $\bm{y},\bm{z}\in\mathbb{R}_{+}^{N}$,
    \begin{equation}\label{eq:PredContract}
    d\bigl(\bm{\Phi}^{h,0}(\bm{y}),\,
           \bm{\Phi}^{h,0}(\bm{z})\bigr)
    \leq
    \sqrt{1 + 2h\nu + h^2L^2}\;
    d(\bm{y},\bm{z}).
    \end{equation}
    In particular, the predictor is non-expansive whenever  $h \leq 2|\nu|/L^2.$ Moreover, if $\nu<0$, contractivity holds for all sufficiently small $h>0$.
\end{theorem}
\begin{proof}
    By construction, when $N_k=0$ and $\tau_0=0$ the SPIDeC method reduces,
    in logarithmic variables, to the explicit Euler scheme $\bm{u}^{n+1} = \bm{u}^n + h\,\bm{G}(\bm{u}^n).$ Let $\bm{u}^n=\operatorname{Log}(\bm{y})$ and $\bm{v}^n=\operatorname{Log}(\bm{z})$, and define
    \[
    \bm{\gamma}^n = \bm{u}^n - \bm{v}^n, 
    \qquad
    \bm{\kappa}^n = \bm{G}(\bm{u}^n) - \bm{G}(\bm{v}^n).
    \]
    Then $\bm{\gamma}^{n+1} = \bm{\gamma}^n + h\,\bm{\kappa}^n.$ Taking the squared norm yields
    \begin{equation*}
        \|\bm{\gamma}^{n+1}\|^2= \|\bm{\gamma}^n\|^2
           + 2h \langle \bm{\gamma}^n, \bm{\kappa}^n \rangle
           + h^2 \|\bm{\kappa}^n\|^2.
    \end{equation*}
    By the one-sided Lipschitz condition~\eqref{eq:OSLC}, $\langle \bm{\gamma}^n, \bm{\kappa}^n \rangle\leq \nu \|\bm{\gamma}^n\|^2, $
    and by Lipschitz continuity, $\|\bm{\kappa}^n\|
    \leq L \|\bm{\gamma}^n\|.$ Therefore,
    \begin{equation*}
        \|\bm{\gamma}^{n+1}\|^2 \leq (1 + 2h\nu + h^2L^2)\|\bm{\gamma}^n\|^2.
    \end{equation*}
    Taking square roots and recalling that $\|\bm{\gamma}^n\| = d(\bm{y},\bm{z})$ concludes the proof.
\end{proof}

A straightforward application of the convergence results of Section \ref{sec:Error_Analysis} leads to the following theorem.
\begin{theorem}[Contractivity after $N_k$ correction 
               sweeps]
\label{thm:NkContract}
    Let the assumptions of 
    Theorem~\ref{thm:NonlinContr} hold and let 
    $N_k\leq M$. Then, there exists a constant $C>0$, 
    independent of $h$, such that for any 
    $\bm{y},\bm{z}\in\mathbb{R}_{+}^{N}$,
    \begin{equation}\label{eq:NkContract}
        \operatorname{d}\bigl(
          \bm{\Phi}^{h,N_k}(\bm{y}),\,
          \bm{\Phi}^{h,N_k}(\bm{z})
        \bigr)
        \leq \operatorname{d}(\bm{y},\bm{z}) 
             + C\,h^{N_k+2}.
    \end{equation}
\end{theorem}

\begin{proof}
    By the triangle inequality,
    \begin{equation}\label{eq:TriangleNk}
        \operatorname{d}\bigl(
          \bm{\Phi}^{h,N_k}(\bm{y}),\,
          \bm{\Phi}^{h,N_k}(\bm{z})
        \bigr)
        \leq
        \operatorname{d}\bigl(
          \bm{\Phi}^{h,N_k}(\bm{y}),\,
          \bm{\Psi}^h(\bm{y})
        \bigr)
        +\operatorname{d}\bigl(
          \bm{\Psi}^h(\bm{y}),\,
          \bm{\Psi}^h(\bm{z})
        \bigr)
        +\operatorname{d}\bigl(
          \bm{\Psi}^h(\bm{z}),\,
          \bm{\Phi}^{h,N_k}(\bm{z})
        \bigr).
    \end{equation}
    By Theorem~\ref{thm:NonlinContr}, the middle term 
    satisfies 
    $\operatorname{d}(\bm{\Psi}^h(\bm{y}),
                     \bm{\Psi}^h(\bm{z}))
    \leq\operatorname{d}(\bm{y},\bm{z})$.
    For the first and third terms, we use the local 
    truncation error estimate of 
    Theorem~\ref{thm:Consistency}: since 
    $\bm{\Psi}^h(\bm{y})$ is the collocation solution 
    (the limit of the DeC iteration as 
    $N_k\to\infty$), its local error with respect to 
    the exact ODE flow is $O(h^{M+2})$, while 
    $\bm{\Phi}^{h,N_k}(\bm{y})$ has local error 
    $O(h^{N_k+2})$ (cf.~\eqref{eq:Local_error} and 
    Theorem~\ref{thm:Consistency}). Notice that  the equivalence of $\operatorname{d}$ and $\|\cdot\|$ on the compact set $\Omega$ allows us to convert the  Euclidean error estimate of Theorem~\ref{thm:Consistency} into a bound in the log-Euclidean metric. Therefore, by the 
    triangle inequality applied to the exact solution 
    $\bm{y}(t_{n+1})$,
    \[
        \operatorname{d}\bigl(
          \bm{\Phi}^{h,N_k}(\bm{y}),\,
          \bm{\Psi}^h(\bm{y})
        \bigr)
        \leq
          \operatorname{d}\bigl(
            \bm{\Phi}^{h,N_k}(\bm{y}),\,
            \bm{y}(t_{n+1})
          \bigr)
        +
          \operatorname{d}\bigl(
            \bm{y}(t_{n+1}),\,
            \bm{\Psi}^h(\bm{y})
          \bigr)
        = \mathcal{O}(h^{N_k+2})+\mathcal{O}(h^{M+2}) \leq C_y\,h^{N_k+2},
    \]
    for some constant $C_y>0$ independent of $h$,  where the last inequality uses $N_k\leq M$.  An analogous bound holds for $\bm{z}$ with constant $C_z$. Substituting into~\eqref{eq:TriangleNk} and setting $C=C_y+C_z$  gives~\eqref{eq:NkContract}.
\end{proof}

The comparison between Theorems~\ref{thm:NonlinContr}, ~\ref{thm:PredContr} and \ref{thm:NkContract} highlights the role of correction sweeps in non-linear stability. The predictor scheme ($N_k=0$) is contractive only under a stepsize restriction depending on the dissipativity and Lipschitz constants. In contrast, the full collocation map $\bm{\Psi}^h$ ($N_k\to \infty$) is unconditionally contractive for any $h>0$. Each correction sweep reduces the approximation gap between $\bm{\Phi}^{h,N_k}$ and $\bm{\Psi}^h$ by one order in $h$, as quantified by~\eqref{eq:NkContract}, progressively inheriting the unconditional contractivity of the collocation map. Hence, increasing the number of correction sweeps simultaneously improves the order of accuracy and relaxes the stepsize restriction for non-linear stability. These two effects are not in competition, but arise from the same deferred correction mechanism.

\section{Numerical Experiments}\label{sec:Numerics}
In this section, we present a series of numerical experiments aimed at validating the theoretical findings of the previous sections. Representative test cases are discretized and numerically integrated using the SPIDeC scheme \eqref{eq:SPIDeC}, based on Gauss--Lobatto (GL) and Gauss--Radau (GR) quadrature nodes. The notation SPIDeC-GL($p$) and SPIDeC-GR($p$) denotes the corresponding methods of order $p$, obtained with the minimum number of points $M+1=p$ and correction sweeps $N_k=p-1$. All simulations are conducted using MATLAB R2025b on an Intel(R) Core(TM) i9-14900KF processor operating at 3200 MHz.

The \texttt{Chebfun} environment~\cite{Driscoll2014ChebfunGuide} is employed for the computation of quadrature nodes and for the assembly of the matrix $Q$ described in Section~\ref{subsec:Quadrature}. GL nodes are obtained via the \texttt{lobpts} routine, which generates $M+1$ points in $[-1,1]$ corresponding to the zeros of $(1-x^2)P'_M(x)$, where $P_M(x)$ denotes the Legendre polynomial of degree $M$. The nodes are then mapped to $[0,1]$ through the affine transformation
\begin{equation}
    \tau_m = \frac{\xi_m + 1}{2}, \qquad m=0,\dots,M,
\end{equation}
which guarantees the inclusion of both boundary points. GR nodes are computed using the \texttt{radaupts} routine, which returns the Legendre--Radau points $\{\zeta_m\}_{m=0}^M \subset [-1,1)$, with the left endpoint fixed at $\zeta_0=-1$. To obtain a quadrature rule including the right boundary of the integration interval, the nodes are mapped to $[0,1]$ via reflection and scaling,
\begin{equation}
    \tau_m = \frac{1 - \zeta_{M-m}}{2}, \qquad m=0,\dots,M.
\end{equation}
This construction places the fixed Radau node at $\tau_M=1$. The use of \texttt{Chebfun} ensures high accuracy in the computation of nodes and of the associated quadrature matrices, which is essential for maintaining the consistency of the SPIDeC discretization.

\begin{testcase}[Experimental order of convergence]\label{test:Convergence}

As a first test, we consider the replicator system
\begin{equation}\label{eq:replicator}
    y_i^\prime(t) = y_i(t)\left( f_i(\bm{y}(t)) - \sum_{j=1}^{N} y_j(t) f_j(\bm{y}(t)) \right),\qquad \qquad i=1,\dots,N,
\end{equation}
where $\bm{y}(t)=[y_1(t),\dots,y_N(t)]^\top\in \mathbb{R}^N$ denotes the state vector and $f_i(\bm{y}(t))$ represents the fitness associated with the $i$-th component, for $i=1,\dots,N$.
It is well known that the $(N-1)$-dimensional probability simplex $\Delta^{N-1}\subset \mathbb{R}^N$ and its interior $\operatorname{int}(\Delta^{N-1})$, defined as
\begin{equation*}
    \Delta^{N-1} = \left\{\bm{\xi} \in \mathbb{R}^N : \sum_{i=1}^{N}\xi_i = 1,\ \xi_i \ge 0,\ i=1,\dots,N \right\}, \qquad \operatorname{int}(\Delta^{N-1})= \Delta^{N-1}\cap \mathbb{R}^N_+,
\end{equation*}
are positively invariant under the flow of the replicator dynamics. In particular, if $\bm{y}^0 \in \operatorname{int}(\Delta^{N-1})$, then $y_i(t) > 0$ for all $t > 0$ and for all $i=1,\dots,N$.
From a numerical viewpoint, the geometric numerical integration of \eqref{eq:replicator} has been investigated in \cite{Replicator_Pezzella}, where a second order explicit rational scheme is proposed. 

Here, to assess the convergence of the SPIDeC method \eqref{eq:SPIDeC}, we consider the specific case of \eqref{eq:replicator} with fitness functions $f_i(\bm{y})=\mathrm{f}_i,$ for which the exact solution is provided as 
\begin{equation}\label{eq:exact_sol}
    y_i(t) = \frac{y_i^0 \exp(\mathrm{f}_i t)}{\sum_{j=1}^{N} y_j^0 \exp(\mathrm{f}_j t)}, \qquad i=1,\dots,N.
\end{equation}
This closed-form expression allows for a precise evaluation of the numerical error $E(h)$ and of the Experimental Order of Convergence (EOC) defined as follows
\begin{equation*}
    E(h)=\dfrac{1}{N_t}\sum_{n=0}^{N_t}\|\bm{y}^n-\bm{y}(t_n)\|_\infty, \qquad \qquad EOC=\log_2\left(\dfrac{E(2h)}{E(h)}\right).
\end{equation*}
For the convergence analysis, we set $N=4$ and consider the fitness vector $\bm{\mathrm{f}} = [15, 5, -10, 20]^\top$. The initial condition is chosen as $\bm{y}^0 = \frac{1}{40}[7, 11, 9, 13]^\top \in \operatorname{int}(\Delta^3)$, and the system is integrated up to the final time $T=1$. The results, reported in Tables \ref{tab:GL_full} and \ref{tab:GR_full} and illustrated in Figure \ref{fig:Ar_Test1_Order_Test}, confirm the theoretical findings of Theorem \ref{thm:Convergence_SPIDeC}. In particular, independently of the choice of quadrature nodes, the experimental order of convergence matches the expected one. Specifically, both SPIDeC implementations with $p=M+1$ quadrature points and $N_k=p-1$ correction sweeps achieve order $p$ for $2\leq p \leq 8.$

\begin{table}[htb]
    \centering
    \caption{Numerical error and EOC for SPIDeC-GL($p$) on the replicator problem. Here, $M=N_k=p-1.$}\label{tab:GL_full}
    \scriptsize 
    \setlength{\tabcolsep}{3pt}
    \begin{tabular}{l |cc |cc |cc |cc |cc |cc |cc|}
    \toprule
    \multicolumn{1}{c}{} & \multicolumn{2}{c}{$p=2$} & \multicolumn{2}{c}{$p=3$} & \multicolumn{2}{c}{$p=4$} & \multicolumn{2}{c}{$p=5$} & \multicolumn{2}{c}{$p=6$} & \multicolumn{2}{c}{$p=7$} & \multicolumn{2}{c}{$p=8$} \\
    \cmidrule(lr){2-3} \cmidrule(lr){4-5} \cmidrule(lr){6-7} \cmidrule(lr){8-9} \cmidrule(lr){10-11} \cmidrule(lr){12-13} \cmidrule(lr){14-15}
    $h$ & Error & EOC & Error & EOC & Error & EOC & Error & EOC & Error & EOC & Error & EOC & Error & EOC \\
    \midrule
    $2^{-4}$ & $1.56 \cdot 10^{-2}$ & ---  & $2.03 \cdot 10^{-3}$ & ---  & $4.62 \cdot 10^{-4}$ & ---  & $7.82 \cdot 10^{-5}$ & ---  & $1.19 \cdot 10^{-5}$ & ---  & $1.61 \cdot 10^{-6}$ & ---  & $1.97 \cdot 10^{-7}$ & ---  \\
    $2^{-5}$ & $1.95 \cdot 10^{-3}$ & 3.00 & $1.89 \cdot 10^{-4}$ & 3.42 & $1.93 \cdot 10^{-5}$ & 4.58 & $1.67 \cdot 10^{-6}$ & 5.55 & $1.27 \cdot 10^{-7}$ & 6.56 & $8.57 \cdot 10^{-9}$ & 7.56 & $5.25 \cdot 10^{-10}$ & 8.55 \\
    $2^{-6}$ & $3.82 \cdot 10^{-4}$ & 2.35 & $1.98 \cdot 10^{-5}$ & 3.25 & $9.93 \cdot 10^{-7}$ & 4.28 & $4.27 \cdot 10^{-8}$ & 5.29 & $1.62 \cdot 10^{-9}$ & 6.29 & $5.48 \cdot 10^{-11}$ & 7.29 & $1.68 \cdot 10^{-12}$ & 8.29 \\
    $2^{-7}$ & $8.58 \cdot 10^{-5}$ & 2.16 & $2.26 \cdot 10^{-6}$ & 3.13 & $5.63 \cdot 10^{-8}$ & 4.14 & $1.21 \cdot 10^{-9}$ & 5.14 & $2.29 \cdot 10^{-11}$ & 6.15 & $3.87 \cdot 10^{-13}$ & 7.15 & $5.92 \cdot 10^{-15}$ & 8.15 \\
    $2^{-8}$ & $2.04 \cdot 10^{-5}$ & 2.07 & $2.70 \cdot 10^{-7}$ & 3.07 & $3.35 \cdot 10^{-9}$ & 4.07 & $3.59 \cdot 10^{-11}$ & 5.07 & $3.39 \cdot 10^{-13}$ & 6.07 & $2.88 \cdot 10^{-15}$ & 7.07 & $9.85 \cdot 10^{-17}$ & 5.91 \\
    $2^{-9}$ & $4.97 \cdot 10^{-6}$ & 2.04 & $3.30 \cdot 10^{-8}$ & 3.03 & $2.04 \cdot 10^{-10}$ & 4.04 & $1.09 \cdot 10^{-12}$ & 5.04 & $5.24 \cdot 10^{-15}$ & 6.02 & $1.21 \cdot 10^{-16}$ & 4.57 & --- & --- \\
    $2^{-10}$ & $1.23 \cdot 10^{-6}$ & 2.02 & $4.07 \cdot 10^{-9}$ & 3.02 & $1.26 \cdot 10^{-11}$ & 4.02 & $3.39 \cdot 10^{-14}$ & 5.01 & $1.47 \cdot 10^{-16}$ & 5.16 & --- & --- & --- & --- \\
    $2^{-11}$ & $3.05 \cdot 10^{-7}$ & 2.01 & $5.06 \cdot 10^{-10}$ & 3.01 & $7.83 \cdot 10^{-13}$ & 4.01 & $1.09 \cdot 10^{-15}$ & 4.95 & --- & --- & --- & --- & --- & --- \\
    \bottomrule
    \end{tabular}
\end{table}
\begin{table}[htb]
    \centering
    \caption{Numerical error and EOC for SPIDeC-GR($p$) on the replicator problem. Here, $M=N_k=p-1.$}\label{tab:GR_full}
    \scriptsize
    \setlength{\tabcolsep}{3pt}
    \begin{tabular}{l |cc |cc |cc |cc |cc |cc |cc|}
    \toprule
    \multicolumn{1}{c}{} & \multicolumn{2}{c}{$p=2$} & \multicolumn{2}{c}{$p=3$} & \multicolumn{2}{c}{$p=4$} & \multicolumn{2}{c}{$p=5$} & \multicolumn{2}{c}{$p=6$} & \multicolumn{2}{c}{$p=7$} & \multicolumn{2}{c}{$p=8$} \\
    \cmidrule(lr){2-3} \cmidrule(lr){4-5} \cmidrule(lr){6-7} \cmidrule(lr){8-9} \cmidrule(lr){10-11} \cmidrule(lr){12-13} \cmidrule(lr){14-15}
    $h$ & Error & EOC & Error & EOC & Error & EOC & Error & EOC & Error & EOC & Error & EOC & Error & EOC \\
    \midrule
    $2^{-4}$ & $1.22 \cdot 10^{-2}$ & ---  & $2.06 \cdot 10^{-3}$ & ---  & $4.62 \cdot 10^{-4}$ & ---  & $7.83 \cdot 10^{-5}$ & ---  & $1.19 \cdot 10^{-5}$ & ---  & $1.61 \cdot 10^{-6}$ & ---  & $1.97 \cdot 10^{-7}$ & ---  \\
    $2^{-5}$ & $1.66 \cdot 10^{-3}$ & 2.88 & $1.90 \cdot 10^{-4}$ & 3.44 & $1.93 \cdot 10^{-5}$ & 4.58 & $1.67 \cdot 10^{-6}$ & 5.55 & $1.27 \cdot 10^{-7}$ & 6.56 & $8.57 \cdot 10^{-9}$ & 7.56 & $5.25 \cdot 10^{-10}$ & 8.55 \\
    $2^{-6}$ & $3.35 \cdot 10^{-4}$ & 2.31 & $1.99 \cdot 10^{-5}$ & 3.26 & $9.93 \cdot 10^{-7}$ & 4.28 & $4.27 \cdot 10^{-8}$ & 5.29 & $1.62 \cdot 10^{-9}$ & 6.29 & $5.48 \cdot 10^{-11}$ & 7.29 & $1.68 \cdot 10^{-12}$ & 8.29 \\
    $2^{-7}$ & $7.62 \cdot 10^{-5}$ & 2.14 & $2.26 \cdot 10^{-6}$ & 3.13 & $5.63 \cdot 10^{-8}$ & 4.14 & $1.21 \cdot 10^{-9}$ & 5.14 & $2.29 \cdot 10^{-11}$ & 6.15 & $3.87 \cdot 10^{-13}$ & 7.15 & $5.93 \cdot 10^{-15}$ & 8.15 \\
    $2^{-8}$ & $1.82 \cdot 10^{-5}$ & 2.06 & $2.70 \cdot 10^{-7}$ & 3.07 & $3.35 \cdot 10^{-9}$ & 4.07 & $3.59 \cdot 10^{-11}$ & 5.07 & $3.39 \cdot 10^{-13}$ & 6.07 & $2.89 \cdot 10^{-15}$ & 7.07 & $8.71 \cdot 10^{-17}$ & 6.09 \\
    $2^{-9}$ & $4.46 \cdot 10^{-6}$ & 2.03 & $3.30 \cdot 10^{-8}$ & 3.03 & $2.04 \cdot 10^{-10}$ & 4.04 & $1.09 \cdot 10^{-12}$ & 5.04 & $5.24 \cdot 10^{-15}$ & 6.02 & $1.05 \cdot 10^{-16}$ & 4.79 & --- & --- \\
    $2^{-10}$ & $1.10 \cdot 10^{-6}$ & 2.02 & $4.08 \cdot 10^{-9}$ & 3.02 & $1.26 \cdot 10^{-11}$ & 4.02 & $3.38 \cdot 10^{-14}$ & 5.02 & $1.81 \cdot 10^{-16}$ & 4.85 & --- & --- & --- & --- \\
    $2^{-11}$ & $2.74 \cdot 10^{-7}$ & 2.01 & $5.06 \cdot 10^{-10}$ & 3.01 & $7.83 \cdot 10^{-13}$ & 4.01 & $1.14 \cdot 10^{-15}$ & 4.89 & --- & --- & --- & --- & --- & --- \\
    \bottomrule
    \end{tabular}
\end{table}

\begin{figure}[htb] 
    \centering
    \includegraphics[width=0.8\linewidth]{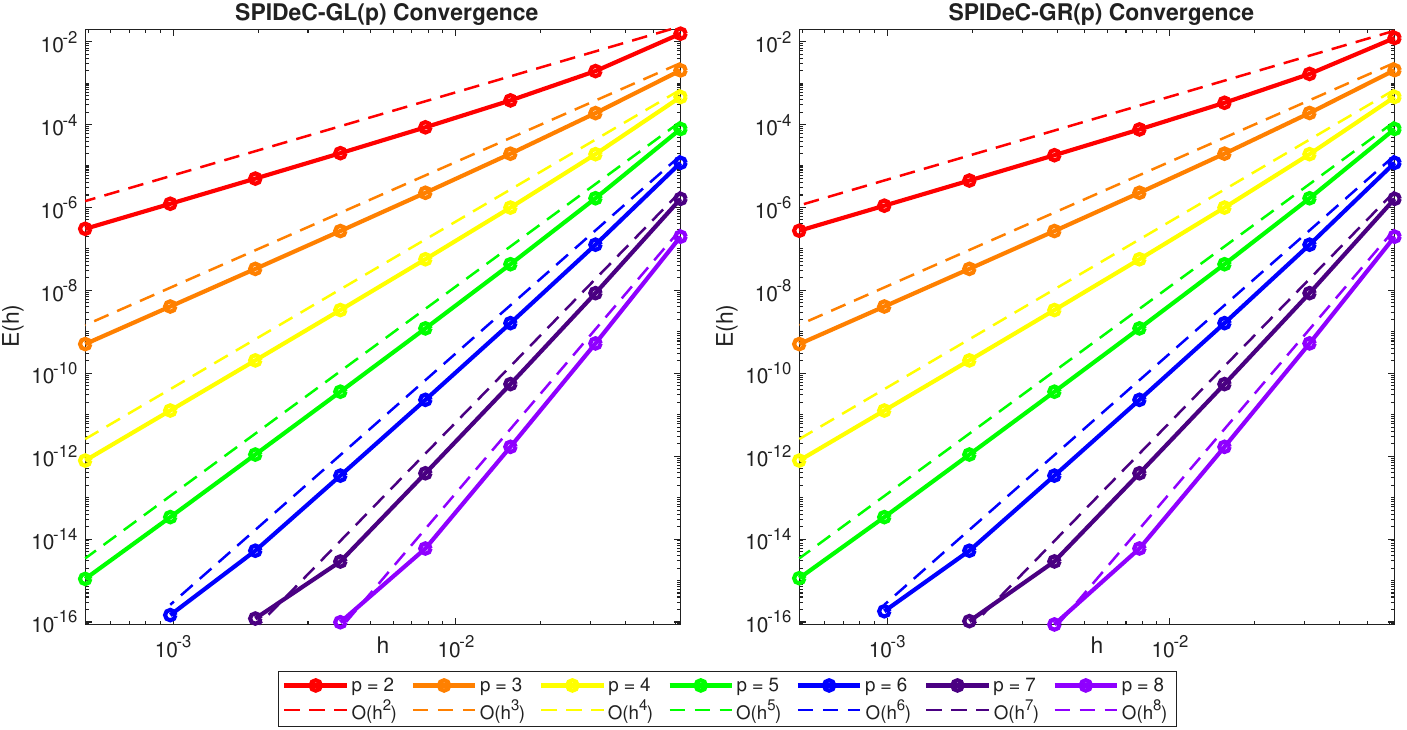}
    \caption{Global numerical error $E(h)$ as a function of the step size $h$ for the SPIDeC-GL and SPIDeC-GR variants applied to Test \ref{test:Convergence}.}\label{fig:Ar_Test1_Order_Test}
\end{figure}

\end{testcase}

\begin{testcase}[Positivity preservation under large stepsizes]
\label{test:positivity}

Theorem~\ref{thm:Positivity_Unc} shows that the SPIDeC framework 
preserves positivity for any $h>0$ and for arbitrary approximation order. 
To illustrate this property in a challenging non-linear regime, 
we consider a predator--prey system with Holling Type~II functional 
responses~\cite{Holling_1959}, defined as in~\eqref{eq:ODE_system} by
\begin{equation}\label{eq:holling}
    f_1(\bm{y}(t)) =  \dfrac{a\varepsilon y_1(t)+(a-b)\,y_1(t)y_2(t)}{\varepsilon + y_2(t)},
    \qquad
    f_2(\bm{y}(t)) =\dfrac{(d-c)\,y_1(t) y_2(t) -c\varepsilon y_2(t)}{\varepsilon + y_1(t)}, \qquad t\in[0,100],
\end{equation}
where $y_1(t)$ and $y_2(t)$ denote the prey and predator population 
densities, respectively. The parameters
\[
a=4.0,\qquad 
b=15.0,\qquad 
c=3.0,\qquad 
d=11.0,\qquad 
\varepsilon=10^{-3}, \qquad \bm{y}^0=[0.02,\,4.0]^\top,
\]
are chosen to induce large-amplitude oscillatory dynamics. The small saturation parameter $\varepsilon=10^{-3}$ makes both 
functional responses nearly singular near the coordinate axes, so that the trajectory repeatedly approaches the boundary of 
$\mathbb{R}^2_{+}$ before recovering. 
This produces a particularly demanding test for positivity-preserving 
time integrators. A high-resolution reference trajectory, computed with 
SPIDeC-GR(10) using $h=10^{-4}$, is reported in 
Figure~\ref{fig:Ar_Test2_Positivity} in logarithmic scales to 
highlight the repeated near-vanishing dynamics of both components.

\begin{figure}[htb]
    \centering
    \includegraphics[width=0.8\linewidth]
    {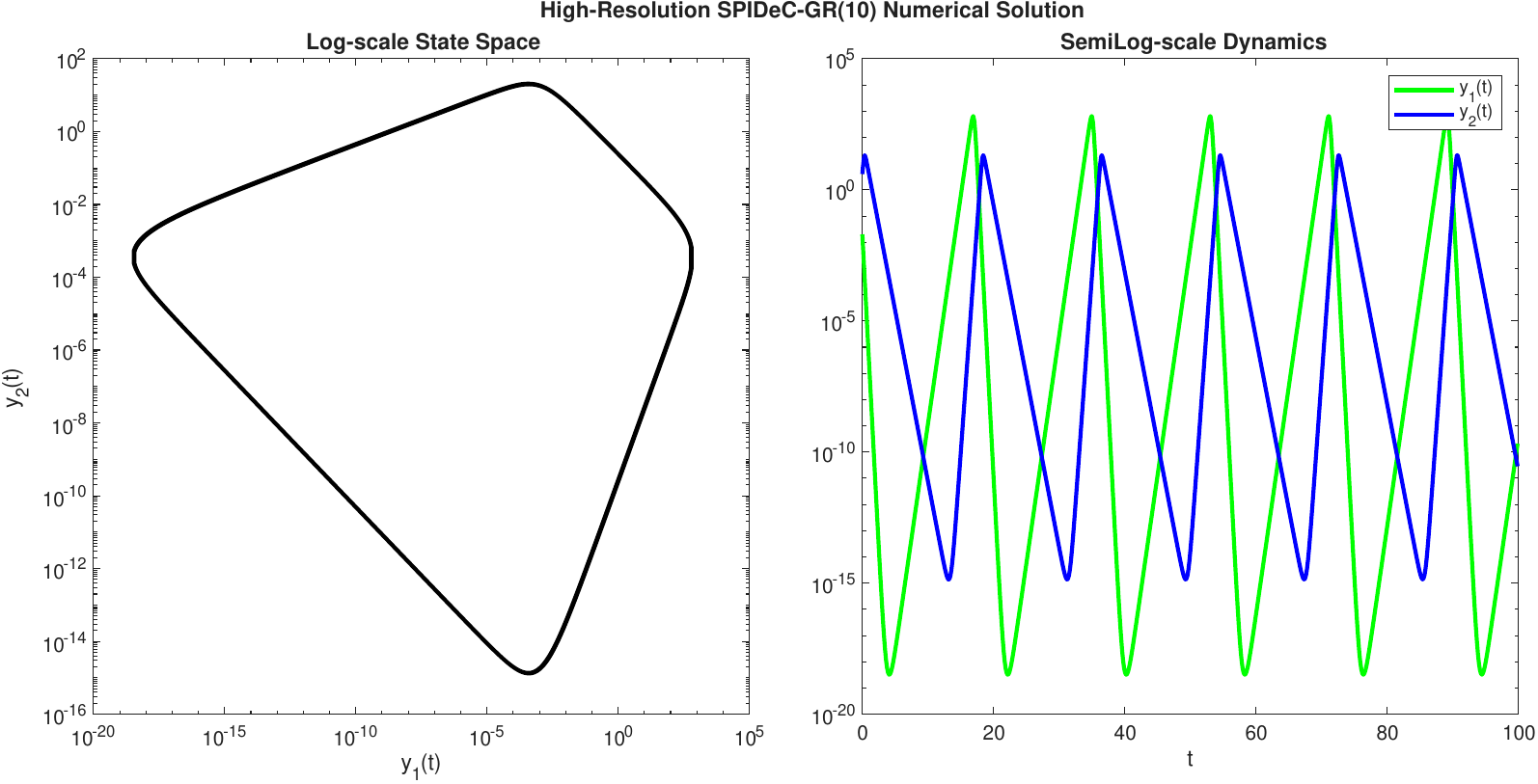}
    \caption{
    High-resolution reference solution of the
    predator--prey system of Test \ref{test:positivity}, computed with 
    SPIDeC-GR(10) using $h=10^{-4}$. 
    The logarithmic representation emphasizes the repeated 
    approach of the trajectory to the boundary of 
    $\mathbb{R}^2_{+}$.
    }
    \label{fig:Ar_Test2_Positivity}
\end{figure}

We compare the SPIDeC-GL($p$) and SPIDeC-GR($p$) methods against 
a representative set of classical explicit Runge--Kutta methods and 
Strong Stability Preserving Runge--Kutta (SSPRK) schemes of orders
$p=2,\dots,6$. 
The comparison metric is the \emph{critical positivity threshold} 
\begin{equation*}
    h^*=\max\left\{h\in \mathcal{H} \ : \ \min_{\substack{n=0,\dots,N_t \\ i=1,\dots,N}}y_i^n(h) > 0 \right\},
\end{equation*}
defined as the largest stepsize in the tested range $\mathcal{H}=[0.1, 5.0]$ for which the numerical solution remains positive throughout the entire integration interval.  Table~\ref{tab:positivity_comparison} summarizes the results. 
All classical explicit methods exhibit positivity thresholds 
restricted to approximately $h^* \lesssim 0.14.$
Increasing the formal approximation order does not enlarge the 
admissible timestep range for positivity preservation, reflecting 
the restrictive non-linear stability constraints of explicit schemes. 
By contrast, both SPIDeC-GL($p$) and SPIDeC-GR($p$) preserve positivity 
for all tested timesteps up to $h=5.0$, independently of the order $p$. 
No visible difference in positivity behavior is observed between 
the Gauss--Lobatto and Gauss--Radau variants, indicating that the 
positivity mechanism is independent of the specific collocation family.

\begin{table}[htb]
    \centering
    \caption{
    Critical positivity threshold $h^*$ for classical explicit methods 
    and SPIDeC applied to the predator--prey 
    system of Test \ref{test:positivity}. 
    Entries in bold indicate unconditional positivity within the 
    tested timestep range.
    }
    \label{tab:positivity_comparison}
    
    \small
    \setlength{\tabcolsep}{6pt}
    
    \begin{tabular}{c | lc | lc}
    \toprule
    Order $p$ 
    & Classical method & $h^*$ 
    & SPIDeC method & $h^*$ \\
    \midrule
    
    $2$  
    & Heun~\cite[Table 1.1]{Hairer_I} 
    & $0.09$ 
    & \textbf{SPIDeC-GL(2)} 
    & $\bm{5.00}$ \\
    
    & SSPRK2~\cite[Equation 2.15]{ShuOsher_1988} 
    & $0.09$ 
    & \textbf{SPIDeC-GR(2)} 
    & $\bm{5.00}$ \\
    
    \midrule
    
    $3$
    & SSPRK3~\cite[Equation 2.18]{ShuOsher_1988} 
    & $0.09$ 
    & \textbf{SPIDeC-GL(3)} 
    & $\bm{5.00}$ \\
    
    & RK3 ~\cite[Table 1.1]{Hairer_I}    
    & $0.14$ 
    & \textbf{SPIDeC-GR(3)} 
    & $\bm{5.00}$ \\
    
    \midrule
    
    $4$
    & RK4~\cite[Table 2.1]{Hairer_I}           
    & $0.11$ 
    & \textbf{SPIDeC-GL(4)} 
    & $\bm{5.00}$ \\
    
    & Ralston4~\cite[Equations (5.11)-(5.12)]{Ralston_1962}  
    & $0.12$ 
    & \textbf{SPIDeC-GR(4)} 
    & $\bm{5.00}$ \\
    
    \midrule
    
    $5$
    & DOPRI5~\cite[Table 2]{DormandPrince_1980}  
    & $0.10$ 
    & \textbf{SPIDeC-GL(5)} 
    & $\bm{5.00}$ \\
    
    & Fehlberg5~\cite[Table III]{Fehlberg_1969}    
    & $0.10$ 
    & \textbf{SPIDeC-GR(5)} 
    & $\bm{5.00}$ \\
    
    \midrule
    
    $6$
    & Luther6~\cite[Equation (3)]{Luther_1968}   
    & $0.05$ 
    & \textbf{SPIDeC-GL(6)} 
    & $\bm{5.00}$ \\
    
    & Verner6~\cite[Table 5.4]{Hairer_I}   
    & $0.10$ 
    & \textbf{SPIDeC-GR(6)} 
    & $\bm{5.00}$ \\
    
    \bottomrule
    \end{tabular}
\end{table}

The experiment numerically confirms the structure-preserving 
mechanism of the SPIDeC framework.
In particular, positivity is retained uniformly with respect to the 
approximation order and remains effective even in highly non-linear 
regimes characterized by near-singular dynamics and large timesteps.

\end{testcase}

\begin{testcase}[Exact reproduction of the semigroup generated by a linear operator]\label{test:Linear}

Our third example concerns the diagonal linear system introduced in~\eqref{eq:Diag_Lin_Sys}. Specifically, we consider the configuration
\begin{equation*}
    \lambda_1=\dfrac{\lambda}{4}, \qquad 
    \lambda_2=\dfrac{\lambda}{2}, \qquad 
    \lambda_3=\dfrac{3}{4}\lambda, \qquad 
    \lambda_4=\lambda, \qquad t\in [0,20],
\end{equation*}
with $\lambda\in\mathbb{R}^-$ acting as a stiffness scaling parameter. We consider a sequence of increasingly stiff configurations obtained by imposing the constraint $h\lambda=-100,$
while varying both the timestep $h$ and the parameter $\lambda$. 
The tested pairs are
\[
    (h,\lambda)\in
    \{(10,-10),\, (5,-20), \, (4,-25), \, (2.5,-40),\, 
    (2,-50), \, (1.25,-80), \, (1,-100)\}.
\]
Hence, although the continuous dynamics become progressively stiffer as $|\lambda|$ increases, the product $h\lambda$ remains fixed at a large negative value. This setting simultaneously probes the behavior of the method under large timesteps and highly dissipative regimes.

Figure~\ref{fig:Ar_Test3_Linear} reports the stiffest component of the numerical solutions obtained with SPIDeC-GL($p$), for $p=2,\dots,5$, together with the corresponding exact exponential decays. The numerical trajectories are visually indistinguishable from the exact semigroup evolution over the entire integration interval, even in the presence of extremely large negative values of $h\lambda$. To make the comparison observable at graphical resolution, the exact exponential profiles are represented by dashed curves shifted by a constant factor $10^{-9}$. This artificial offset avoids the complete overlap between the numerical and exact solutions while preserving the exact decay rate, which is reflected by the identical slope of the corresponding curves in the logarithmic scale.

A quantitative confirmation of the exact preservation property established in Theorem~\ref{Thm:Linear_Stability} is obtained by evaluating the maximum approximation error
\begin{equation*}
     E_\infty(h,\lambda)
    =
    \max_{n=0,\dots,N_t}
    \left|
        y_4^n-\exp(\lambda t_n)
    \right|,
\end{equation*}
whose values are reported in Table \ref{tab:linear_semigroup_roundoff}. We observe that $E_\infty(h,\lambda)$ remains extremely small across all tested configurations, reaching values of the order of $10^{-57}$ in the least favourable case. More importantly, the error is invariant with respect to the pair $(h,\lambda)$ and depends only on the order $p$, hence on the number of floating-point operations involved. These results confirm the exact semigroup preservation property stated in Theorem~\ref{Thm:Linear_Stability}. The observed discrepancy is entirely due to round-off accumulation in finite precision arithmetic and does not reflect any discretization error of the scheme.

\begin{figure}[htb]
    \centering
    \includegraphics[width=1\linewidth]
    {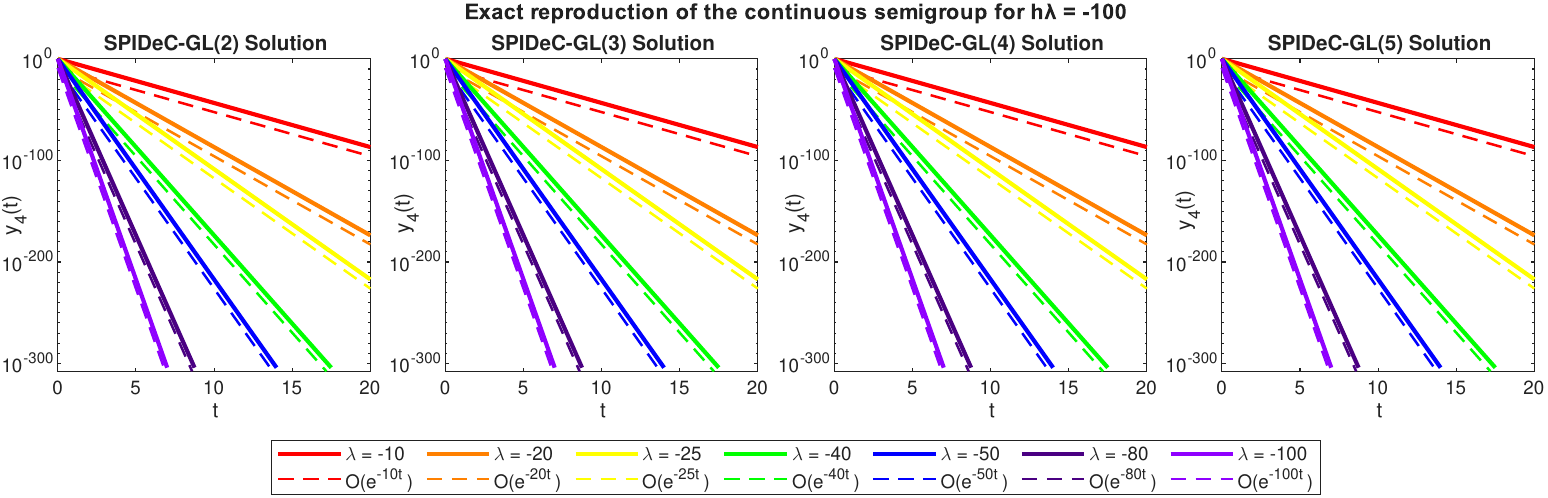}
    \caption{
    Stiff component of the SPIDeC-GL($p$) numerical solution of Test \ref{test:Linear} compared with the exact exponential decay. Exact solutions are represented by dashed curves with a vertical shift of $10^{-9}$ for visual clarity.
    }
    \label{fig:Ar_Test3_Linear}
\end{figure}
\begin{table}[htb]
    \centering
    \caption{
    Maximum error $E_\infty(h,\lambda)$ for SPIDeC-GL($p$) applied to the linear diagonal Test \ref{test:Linear}.
    }
    \label{tab:linear_semigroup_roundoff}
    \small
    \setlength{\tabcolsep}{8pt}
    \begin{tabular}{cc|cccc}
    \toprule
    $h$ & $-\lambda$ & $p=2$ & $p=3$ & $p=4$ & $p=5$ \\
    \midrule
    $10$ & $10$   & $4.98\cdot10^{-60}$ & $5.38\cdot10^{-58}$ & $1.05\cdot10^{-57}$ & $5.38\cdot10^{-58}$ \\
    $5$ & $20$    & $4.98\cdot10^{-60}$ & $5.38\cdot10^{-58}$ & $1.05\cdot10^{-57}$ & $5.38\cdot10^{-58}$ \\
    $4$ & $25$    & $4.98\cdot10^{-60}$ & $5.38\cdot10^{-58}$ & $1.05\cdot10^{-57}$ & $5.38\cdot10^{-58}$ \\
    $2.5$ & $40$  & $4.98\cdot10^{-60}$ & $5.38\cdot10^{-58}$ & $1.05\cdot10^{-57}$ & $5.38\cdot10^{-58}$ \\
    $2$ & $50$    & $4.98\cdot10^{-60}$ & $5.38\cdot10^{-58}$ & $1.05\cdot10^{-57}$ & $5.38\cdot10^{-58}$ \\
    $1.25$ & $80$  & $4.98\cdot10^{-60}$ & $5.38\cdot10^{-58}$ & $1.05\cdot10^{-57}$ & $5.38\cdot10^{-58}$ \\
    $1$ & $100$    & $4.98\cdot10^{-60}$ & $5.38\cdot10^{-58}$ & $1.05\cdot10^{-57}$ & $5.38\cdot10^{-58}$ \\
    \bottomrule
    \end{tabular}
\end{table}
\end{testcase}

\begin{testcase}[Non-linear logarithmic contractivity]
\label{test:contractivity}

In order to provide numerical validation of the non-linear stability findings of Section~\ref{subsec:NonlinStab}, we consider the two-dimensional system
\begin{equation}\label{eq:test_contract_sys}
    f_1(\bm{y}(t)) = -\alpha y_1(t) + \beta\sin(y_2(t)) + e^{-y_1^2(t)},
    \qquad
    f_2(\bm{y}(t)) = -\alpha y_2(t) - \beta\sin(y_1(t)) + e^{-y_2^2(t)},
    \qquad
    \alpha = 1.2,\quad \beta = 0.4,
\end{equation}
with initial conditions $\bm y^0=[1.8,\,0.7]^\top$ and $\bm z^0=[1.6,\,1.1]^\top.$ Since the vector field remains strictly inward-pointing along both coordinate axes, i.e. 
\begin{equation*}
    \lim_{y_i\to0^+}f_i(y_i,y_j)\geq 1-\beta >0, \qquad i\neq j\in \{1,2\},
\end{equation*}
the positive orthant is invariant for the dynamics of \eqref{eq:test_contract_sys} and the Assumption \ref{ass:Positivity} holds. Introducing the logarithmic transformation \eqref{eq:LogSystem} and evaluating the Jacobian matrix $J_{\bm G}$ locally at $\bm e=[1,1]^\top$ yields the following estimates of the logarithmic Lipschitz constant
and the dissipativity coefficient
\begin{equation*}
    L=\|J_{\bm G}(\bm y_{\bm{e}})\|_2\approx1.461,
    \qquad
    \nu=
    \lambda_{\max}\!\left(
    \frac{
    J_{\bm G}(\bm y_{\bm{e}})
    +
    J_{\bm G}(\bm y_{\bm{e}})^\top
    }{2}
    \right)\approx -0.767.
\end{equation*}

As a first step, we investigate the contractivity of the predictor scheme established in Theorem \ref{thm:PredContr}. For different values of $h\in[10^{-5},10^2],$ we compute the logarithmic distance ratio, or numerical contraction factor, as follows 
\begin{equation*}
    \operatorname{contr}(h)=\frac{\operatorname d(\bm\Phi^{h,0}(\bm y),\bm\Phi^{h,0}(\bm z))}{\operatorname d(\bm y,\bm z)},
\end{equation*}
where $\operatorname d$ denotes the log-Euclidean metric defined in \eqref{eq:distance}. Here, the predictor map $\bm\Phi^{h,0}$ in \eqref{eq:Phi} is computed performing a single step of the SPIDeC-GL($1$) scheme. The results are shown in the left panel of Figure \ref{fig:Ar_Test4_Predictor}. From the experiments, $\operatorname{contr}(h)$ remains below unity for small values of $h$ and crosses the threshold near the theoretically predicted value $h_{\nu}=2 |\nu|L^{-2}\approx0.72$, in close agreement with the estimate of Theorem~\ref{thm:PredContr}. For larger timesteps, the predictor loses contractivity, thus motivating the introduction of correction sweeps. In any case, the theoretical estimate \eqref{eq:PredContract} accurately bounds the observed contraction factor over the entire range of tested timesteps. 

To address the convergence of the DeC iteration towards the collocation map in \eqref{eq:CollocationMap} (cf. Theorem~\ref{thm:NkContract}), we consider SPIDeC-GR methods with the non-contractive stepsize $h=2.$ Specifically, we monitor the logarithmic distance gap
\begin{equation*}
    \operatorname{gap}(N_k)=\Bigl|\operatorname d(\bm\Phi^{h,N_k}(\bm y),\bm\Phi^{h,N_k}(\bm z))-\operatorname d(\bm\Psi^h(\bm y),\bm\Psi^h(\bm z))\Bigr|,
\end{equation*}
where the collocation map $\bm\Psi^h$ is approximated by $\bm\Phi^{h,M},$ as defined in \eqref{eq:Phi}, with $M=200$. The results, reported in the right panel of Figure \ref{fig:Ar_Test4_Predictor}, show that the gap decay is geometric and that the experimental convergence rate as $N_k\to \infty$ is sharper than the conservative estimate provided by Theorem~\ref{thm:NkContract}. 

\begin{figure}[htb]
    \centering
     \includegraphics[width=0.8\linewidth]{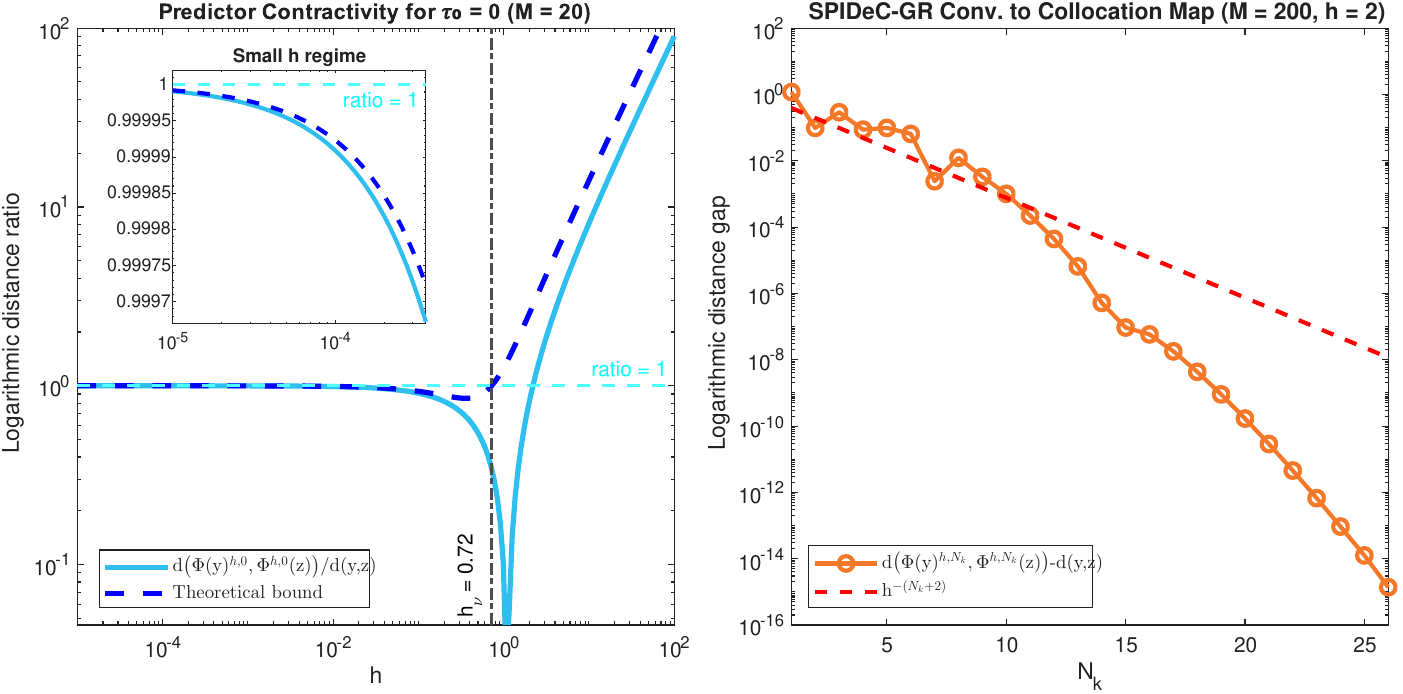}
    \caption{non-linear contractivity Test \ref{test:contractivity}. 
             \textit{Left}: logarithmic distance ratio as a function of the stepsize for the predictor ($N_k=0$, 
             $M=20$ GL nodes), compared with the 
             theoretical bound $\sqrt{1+2h\nu+h^2L^2}$ 
             of Theorem~\ref{thm:PredContr}. 
             \textit{Right}: logarithmic distance gap as a function of the correction sweep ($h=2$, $M=200$ GR nodes) 
             with the geometric reference $h^{-(N_k+2)}$ 
             of Theorem~\ref{thm:NkContract}.}
    \label{fig:Ar_Test4_Predictor}
\end{figure}

Our final experiment aims to clarify how the successive correction sweeps progressively transfer the contractive properties of the collocation map (cf. Theorem \ref{thm:NonlinContr}) to the SPIDeC approximation. To this end, we simulate the system \eqref{eq:test_contract_sys} over the interval $[0,50]$ with $h=2$, $M=4$ and $N_k\in\{12,14,16,18,20,22,24\}$. Both Gauss--Lobatto and Gauss--Radau nodes are considered. The long-time behaviour of the logarithmic distance, after several correction levels, is illustrated in Figure~\ref{fig:Ar_Test4_Log_Distance}. There, the reference exponential profile (black line) represents the exact contraction rate of the continuous flow predicted by~\eqref{eq:ContContractivity}. In the Gauss--Radau case (right panel), increasing the number of correction sweeps progressively restores the theoretical decay rate, and for sufficiently large $N_k$ the discrete dynamics closely reproduce the exact contractive behaviour. By contrast, the Gauss--Lobatto discretization (left panel) exhibits a slower decay and does not recover the asymptotic rate uniformly, reflecting the lack of algebraic stability of the associated Lobatto~IIIA collocation method.

\begin{figure}[htb]
    \centering
    \includegraphics[width=0.8\linewidth]{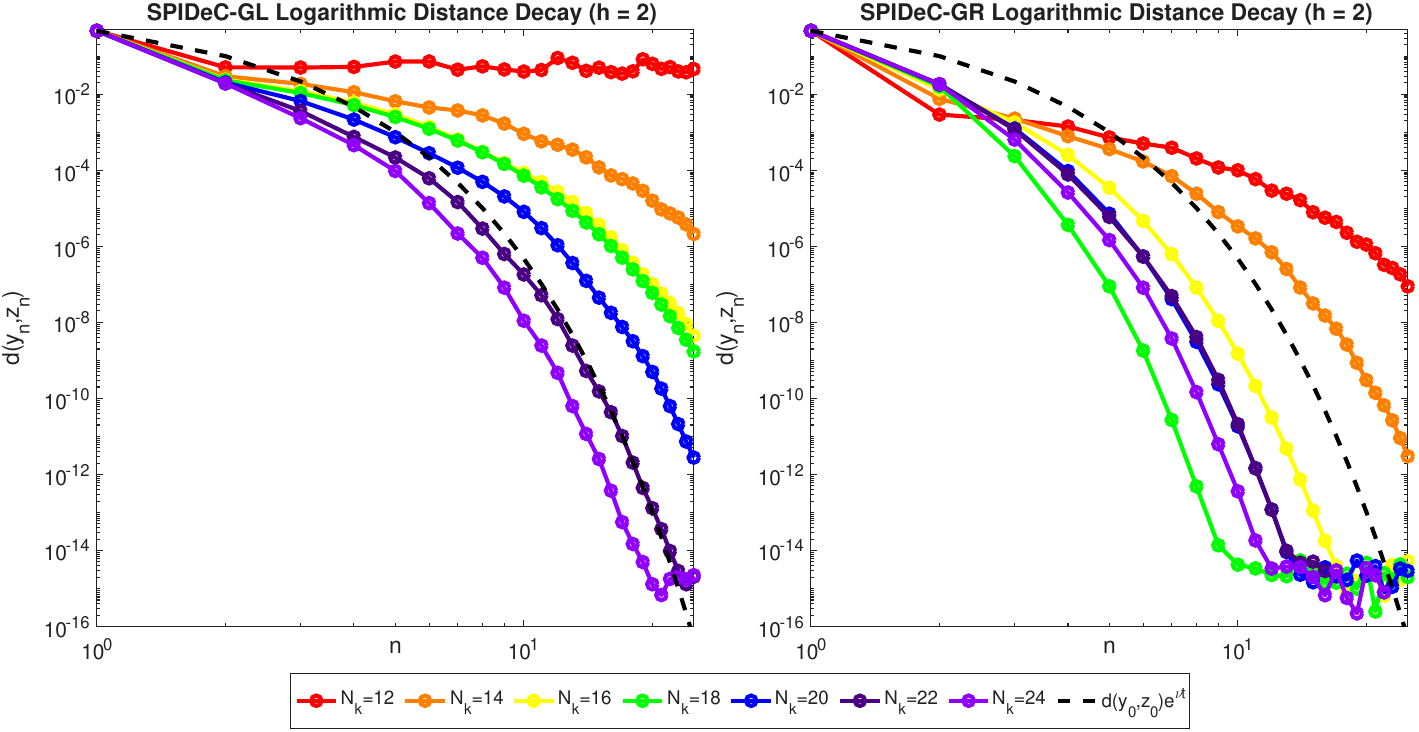}
    \caption{Long-time log-Euclidean distance 
             $\operatorname{d}(\bm{y}^n,\bm{z}^n)$ 
             for SPIDeC-GL (left) and SPIDeC-GR 
             (right) with $h=2$, $M=4$, and 
             $N_k\in\{12,14,\dots,24\}$. 
             The dashed curve 
             $\operatorname{d}(\bm{y}^0,\bm{z}^0)\,
             e^{\nu t}$ represents the theoretical 
             contraction rate~\eqref{eq:ContContractivity}. 
             }
    \label{fig:Ar_Test4_Log_Distance}
\end{figure}

Overall, the positive system \eqref{eq:test_contract_sys} provides a comprehensive validation of the non-linear logarithmic stability framework developed in Section~\ref{subsec:NonlinStab}. The predictor iteration is contractive only under a suitable time-step restriction, correction sweeps reduce the discrepancy with the collocation map at a geometric rate, and sufficiently accurate Gauss--Radau DeC approximations recover the unconditional logarithmic contractivity of the continuous flow, together with the corresponding asymptotic decay rate.
\end{testcase}

\begin{testcase}[Pattern formation in a chemotaxis-driven PDE system]
\label{test:PDE}
 
The role of structure-preserving numerical integration in pattern-forming systems has been widely recognized in the literature (see~\cite{Geco_Eco,Monti2025,SETTANNI} and references therein). Indeed, inappropriate time discretizations may alter the qualitative dynamics of the continuous model and even induce spurious Turing instabilities in parameter regimes where the homogeneous equilibrium is stable. Motivated by these considerations, we investigate the performance of the SPIDeC methodology on positive pattern-forming PDEs and compare it with the symplectic-based strategy proposed in~\cite{Monti2025}.
 
We consider two reaction--diffusion--chemotaxis models exhibiting distinct classes of chemotaxis-driven spatial patterns. In both cases, the unknown positive functions $u(x,y,t)$ and $v(x,y,t)$ represent the density of a biological population and the concentration of a self-produced chemo-attractant, respectively.

Our first case study concerns the two-dimensional Mimura--Tsujikawa reaction--diffusion--chemotaxis system~\cite{Mimura1996}
\begin{equation}\label{eq:MT}
    \begin{cases}
        \partial_t u(x,y,t) = D_u\,\Delta u(x,y,t) -\beta\,\nabla\!\cdot\!(u(x,y,t)\,\nabla v(x,y,t)) +q\,u(x,y,t)(1-u(x,y,t)), \\ \partial_t v(x,y,t) = D_v\,\Delta v(x,y,t) + k_1\, u(x,y,t) - k_2\, v(x,y,t),
    \end{cases}
    \quad (x,y,t) \in \Omega \times[0,T],
\end{equation}
on the square domain $\Omega=[0,L_x]\times[0,L_y]$ and supplemented with homogeneous Neumann boundary conditions. Following~\cite{Monti2025}, we select the parameter values
\begin{equation*}
    D_u=0.0625,\qquad D_v=1,\qquad \beta=17,\qquad q=7,\qquad k_1=1,\qquad k_2=32,\qquad L_x=L_y=6,\qquad T=600,
\end{equation*}
for which the spatially homogeneous equilibrium state $P^{*}=[1,k_1/k_2]^\top$ undergoes a chemotaxis-driven instability leading to the emergence of hexagonal Turing patterns~\cite{Monti2025}.

As a second benchmark, we consider the two-compartment spatial version of the MOMOS (Modelling Organic changes by Micro-Organisms of Soil) model~\cite{Hammoudi2018},
\begin{equation}\label{eq:MOMOS}
    \begin{cases}
        \partial_t u(x,y,t) = D_u\,\Delta u(x,y,t) -\beta\,\nabla\!\cdot\!(u(x,y,t)\,\nabla v(x,y,t))-k_1 \, u(x,y,t)-q \, u^2(x,y,t)+k_2 \, v(x,y,t) , \\ \partial_t v(x,y,t) = D_v\,\Delta v(x,y,t)+k_1 \, u(x,y,t)-k_2 \, v(x,y,t)+c
    \end{cases}
    \ (x,y,t) \in \Omega \times[0,T],
\end{equation}
considered on the square domain $\Omega=[0,L_x]\times[0,L_y]$ with periodic boundary conditions. The parameters are chosen as
\begin{equation*}
    D_u=0.6,\qquad D_v=0.6,\qquad \beta=1.2,\qquad
    k_1=0.4,\qquad k_2=0.6,\qquad q=0.075,\qquad c=0.8,\qquad
    L_x=L_y=25,\qquad T=10^4.
\end{equation*}
The positive spatially homogeneous equilibrium is given by $P^*=\left[\sqrt{\frac{c}{q}}\,,\frac{k_1\sqrt{c/q}+c}{k_2}\right]^\top.$ The selected parameter regime belongs to the chemotaxis-driven instability region identified in~\cite[Theorem~1]{Monti2025}, where spot-type spatial patterns are expected to emerge.

The PDE systems \eqref{eq:MT} and \eqref{eq:MOMOS} are discretised in space by the Method of Lines on a uniform grid consisting of $80^2$ nodes. Diffusion terms are approximated by second order central finite differences, while the chemotactic flux is discretised using the non-linear finite-difference formulation proposed in~\cite[Appendix~A]{Monti2025}. The resulting semi-discrete problems can be written in the form of~\eqref{eq:ODE_system}, with $N=12800$ and
\begin{equation*}
    \bm{y}=[u_1,\ldots,u_{6400},v_1,\ldots,v_{6400}]^\top\in\mathbb{R}_{+}^{N}.
\end{equation*}
Following the numerical experiments reported in~\cite{Monti2025}, the initial conditions are generated as random perturbations of the corresponding homogeneous equilibrium state $P^*$.
\begin{figure}[htb]
    \centering
    \includegraphics[width=1\linewidth]{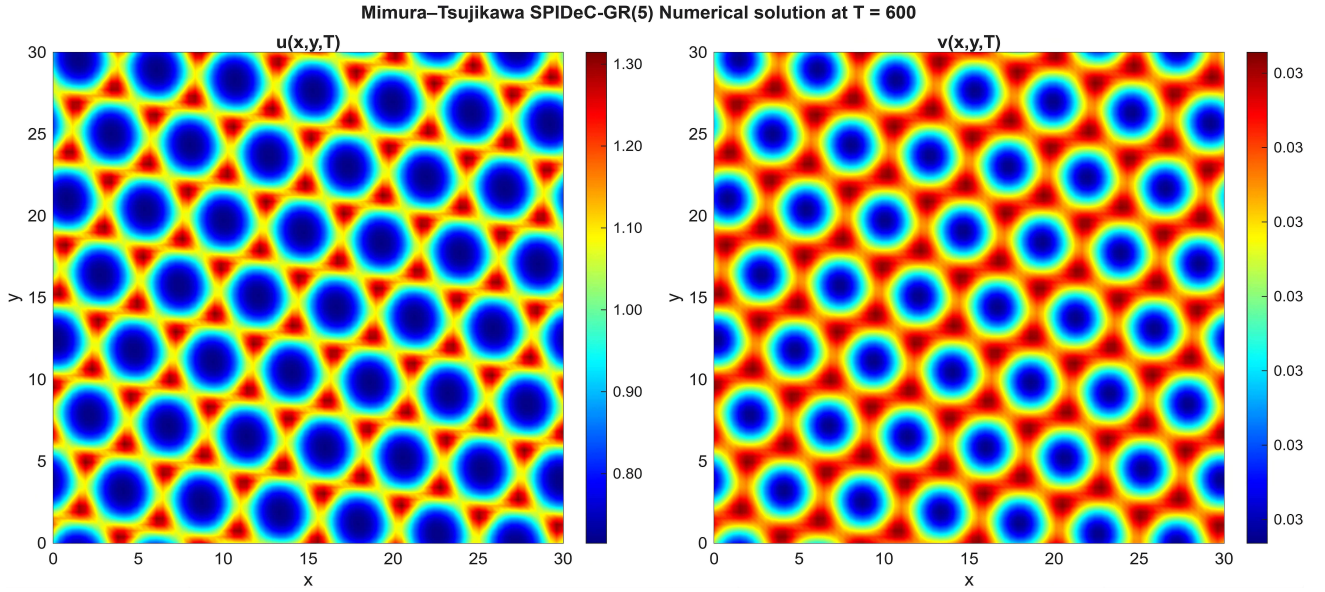}
    \caption{Numerical solution of the Mimura--Tsujikawa system~\eqref{eq:MT} at final time $T=600$, computed via the SPIDeC-GR($5$) scheme with step size $h=10^{-3}$. The spatial distribution illustrates the emergence of the expected hexagonal Turing patterns driven by chemotactic instability.}
    \label{fig:Ar_Test5_MT}
\end{figure}
\begin{figure}[htb]
    \centering
    \includegraphics[width=1\linewidth]{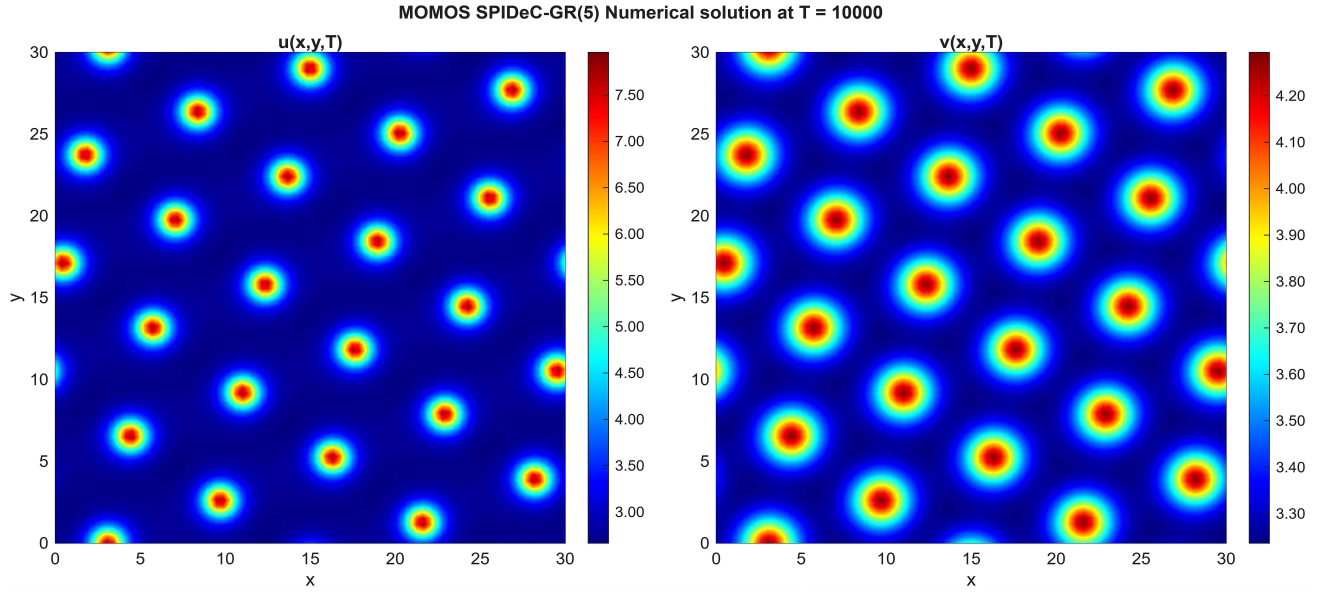}
    \caption{Numerical solution of the MOMOS system~\eqref{eq:MOMOS} at final time $T=10^4$, integrated using the SPIDeC-GR($5$) method with step size $h=10^{-2}$. The simulation accurately captures the formation of spot-type spatial patterns.}
    \label{fig:Ar_Test5_MOMOS}
\end{figure}
Herein, we evaluate the performance of four time integration schemes: the Symplectic Euler (SE) method~\cite[eq.~(11)]{Monti2025}, its semi-implicit variant (IMSP\_IE)~\cite[eq.~(12)]{Monti2025}, and the SPIDeC-GR(p) family for $p=3,5$. The numerical integration is performed using the stepsizes already adopted in~\cite{Monti2025}, namely $h=10^{-3}$ for~\eqref{eq:MT} and $h=10^{-2}$ for~\eqref{eq:MOMOS}. The outcomes corresponding to the SPIDeC-GR(5) scheme are illustrated in Figures~\ref{fig:Ar_Test5_MT} and \ref{fig:Ar_Test5_MOMOS}.
The temporal evolution of both the spatial mean and the Euclidean norm of successive differences, here omitted for brevity, indicates a clear convergence toward steady states. Within this framework, all considered schemes yield geometrically consistent results. However, the higher order SPIDeC methods achieve superior accuracy at the expense of increased computational cost (see Figures~\ref{fig:Mimura_Final_Err} and \ref{fig:Momos_Final_Err} for representations of the final-time spatial error on $u$).
\begin{figure}[htb]
    \centering
    \includegraphics[width=1\linewidth]{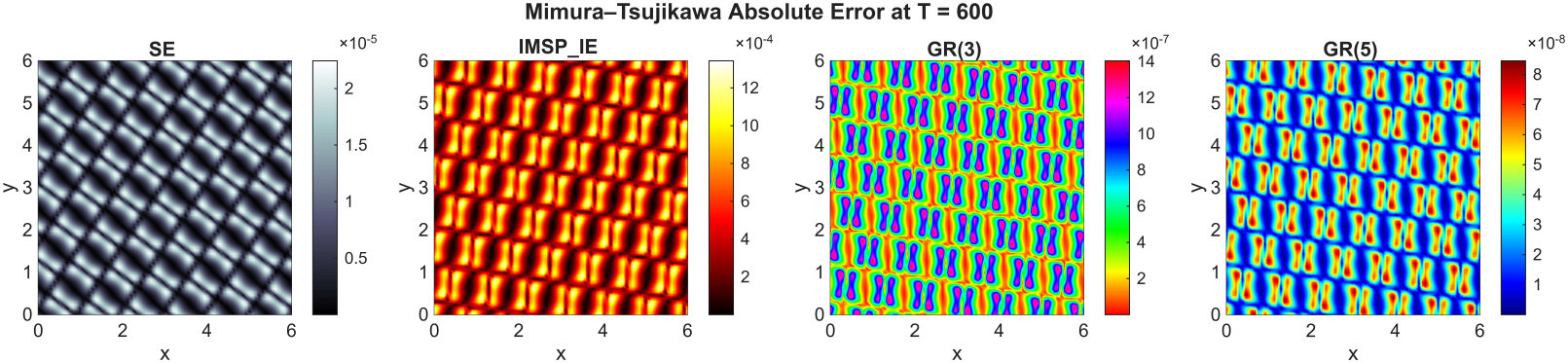}
    \caption{Spatial distribution of the final-time absolute error on the population density $u$ for the Mimura--Tsujikawa model~\eqref{eq:MT}, evaluated against the reference \texttt{ode15s} solution.}
    \label{fig:Mimura_Final_Err}
\end{figure}
\begin{figure}[htb]
    \centering
    \includegraphics[width=1\linewidth]{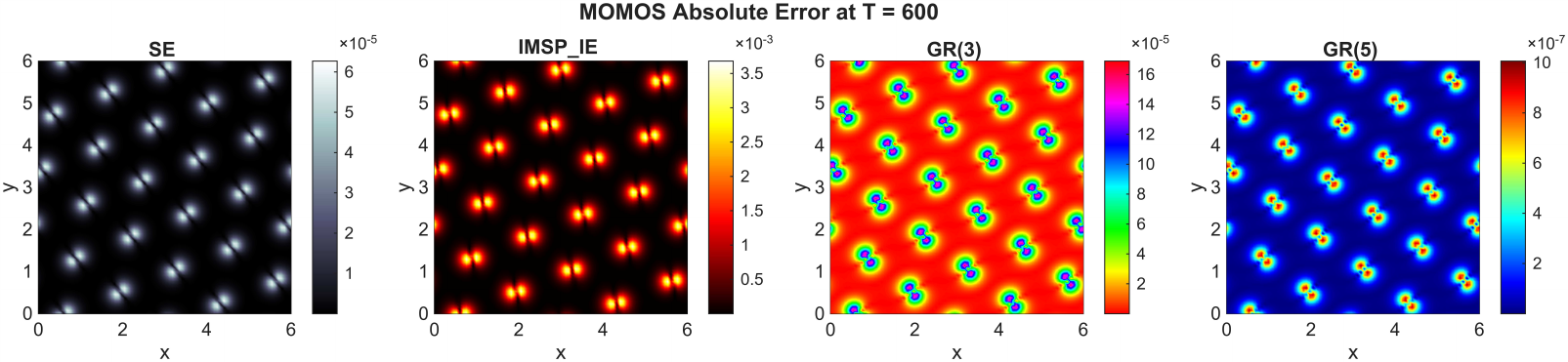}
    \caption{Spatial distribution of the final-time absolute error on the species density $u$ for the MOMOS system~\eqref{eq:MOMOS}, evaluated against the reference \texttt{ode15s} solution.}
    \label{fig:Momos_Final_Err}
\end{figure}

To rigorously assess this accuracy–cost trade-off, the reference solution is computed via MATLAB’s \texttt{ode15s} solver, accelerated through the provision of the analytical Jacobian. We then define the normalized efficiency index
\begin{equation*}
    \eta = \dfrac{\mathcal{E}_{\min}}{\mathcal{E}} \cdot \dfrac{\mathcal{T}_{\min}}{\mathcal{T}},
\end{equation*}
where $\mathcal{E}$ and $\mathcal{T}$ denote the $L^2$ error on $u$ and the CPU time of a specific method, respectively, while $\mathcal{E}_{\min}$ and $\mathcal{T}_{\min}$ represent their minimal values across all tested schemes. By definition, $0 < \eta \le 1$, with $\eta = 1$ representing an ideal theoretical optimum. The comparative findings, detailed in Table~\ref{tab:efficiency_rank}, reveal that for both semi-discretized PDE systems, SPIDeC-GR($5$) and IMSP\_IE stand as the most and the least efficient integrator, respectively. The comparison of SE and SPIDeC-GR($3$) is less sharp, with the latter exhibiting superior efficiency only in the case of the highly stiff Mimura–Tsujikawa model.

\begin{table}[htb]
\centering
\caption{Computational cost, final $L^2$ error, normalized efficiency index $\eta$, and ranking based on $\eta$ (I = best, IV = worst).}
\label{tab:efficiency_rank}
\scriptsize
\setlength{\tabcolsep}{5pt}
\begin{tabular}{lcccc|cccc}
\toprule
& \multicolumn{4}{c|}{\textbf{Mimura--Tsujikawa}} & \multicolumn{4}{c}{\textbf{MOMOS}} \\
\cmidrule(lr){2-5} \cmidrule(lr){6-9}
Method & CPU time & $L^2$ error & $\eta$ & Rank
        & CPU time & $L^2$ error & $\eta$ & Rank \\
\midrule
SE
& $3.44\cdot10^{2}$ & $1.11\cdot10^{-5}$ & $3.57\cdot10^{-3}$ & III
& $6.14\cdot10^{2}$ & $1.28\cdot10^{-5}$ & $1.62\cdot10^{-2}$ & II \\
IMSP\_IE
& $8.60\cdot10^{2}$ & $6.25\cdot10^{-4}$ & $2.53\cdot10^{-5}$ & IV
& $1.14\cdot10^{3}$ & $7.53\cdot10^{-4}$ & $1.48\cdot10^{-4}$ & IV \\
SPIDeC-GR(3)
& $1.48\cdot10^{3}$ & $6.70\cdot10^{-7}$ & $1.37\cdot10^{-2}$ & II
& $1.96\cdot10^{3}$ & $3.49\cdot10^{-5}$ & $1.86\cdot10^{-3}$ & III \\
SPIDeC-GR(5)
& $3.45\cdot10^{3}$ & $3.96\cdot10^{-8}$ & $9.96\cdot10^{-2}$ & I
& $5.69\cdot10^{3}$ & $2.08\cdot10^{-7}$ & $1.08\cdot10^{-1}$ & I \\
\bottomrule
\end{tabular}
\end{table}

\end{testcase}
 
\section{Conclusion}\label{sec:Conclusion}

In this paper, we have developed and analyzed the SPIDeC framework for the numerical integration of positive dynamical systems. By embedding a deferred correction mechanism within an exponential-type Volterra integral reformulation, we obtained a class of integrators that unconditionally preserve both positivity and equilibria, independently of the time step size. This multiplicative structure overcomes the classical order limitations of linear positivity-preserving schemes and allows the systematic construction of arbitrarily high-order methods through explicit-in-sweep correction iterations applied to a low-order base approximation. From a theoretical standpoint, the analysis shows that the resulting SPIDeC integrators are L-stable and exactly reproduce the continuous semigroup generated by diagonal linear operators. Moreover, the use of Gauss–Radau quadrature nodes leads to a discrete flow that becomes increasingly contractive in a logarithmic sense as the number of correction sweeps increases. As a consequence, nonlinear stability properties become progressively stronger. These findings are consistently supported by numerical experiments, which confirm the accuracy, robustness, and efficiency of the proposed approach across representative test problems.

Several directions for future research naturally arise from the present work. The multi-level structure of the SPIDeC correction sweeps suggests the development of adaptive time-stepping strategies based on reliable local truncation error indicators. In such a framework, stepsize selection may be driven by inter-level differences. Another promising direction concerns the quadrature and interpolation operators underlying the Volterra-type reformulation. Beyond the classical Lagrange construction considered here, alternative approximation techniques may improve accuracy for steep functions. In particular, de la Vallée Poussin interpolating polynomials \cite{Woula} and kernel-based rational operators \cite{Woula_Mark} represent promising alternatives. Another interesting research direction concerns the adaptation of the SPIDeC framework to broader classes of ODEs generating order-preserving flows. In particular, systems satisfying the Kamke–M{\"u}ller condition provide a natural setting in which to investigate the construction of high-order schemes capable of preserving positivity and other monotonicity properties at the discrete level. Finally, further investigating the application of the SPIDeC framework to systems of PDEs, through its coupling with structure-preserving spatial discretizations, appears as a natural continuation of the present work.

\begin{acknowledgement}
    {\color{jcolour}\textbf{Acknowledgements}} This work has been supported by the \emph{Italian National Group for Scientific Computing} (GNCS) of the National Institute for Advanced Mathematics (INdAM). This research has been accomplished within the \emph{Research ITalian network on Approximation} (RITA) and the SIMAI   \emph{Numerical and Analytical Approximation of Data and Functions with Applications} (ANA$\&$A) activity group of the \emph{Italian Society of Applied and Industrial Mathematics} (SIMAI). 
\end{acknowledgement}
\begin{acknowledgement}
    {\color{jcolour}\textbf{Data Availability Statement}} The codes implementing the numerical methods discussed in this work are available from the author upon reasonable request. No new data were generated or analyzed in this study and references to existing data are provided in the manuscript.
\end{acknowledgement}

\printbibliography

\end{document}